\documentclass[]{article}

\usepackage{authblk}
\usepackage{amsmath}
\usepackage{amssymb}
\usepackage{amsthm}
\usepackage{hyperref}
\usepackage{cleveref}
\usepackage{paralist}
\usepackage{csquotes}
\usepackage{bm}
\usepackage{subcaption}
\usepackage[backend=bibtex,giveninits=true]{biblatex}
\DeclareMathOperator*{\esssup}{ess\,sup}

\usepackage[colorinlistoftodos,prependcaption,textsize=tiny]{todonotes}

\newcommand{\iu}{\mathrm{i}\mkern1mu}
\newtheorem{theorem}{Theorem}
\newtheorem{lemma}[theorem]{Lemma}
\newtheorem{proposition}[theorem]{Proposition}

\newtheorem{definition}[theorem]{Definition}
\newtheorem{assumption}[theorem]{Assumption}
\newtheorem{remark}[theorem]{Remark}
\newtheorem{problem}[theorem]{Problem}

\crefname{definition}{definition}{definitions}
\Crefname{definition}{Definition}{Definitions}
\Crefname{problem}{Problem}{Problems}
\crefname{problem}{problem}{problems}
\crefname{remark}{remark}{remarks}
\Crefname{remark}{Remark}{Remarks}
\crefname{assumption}{assumption}{assumptions}
\Crefname{assumption}{Assumption}{Assumptions}

\numberwithin{theorem}{section}
\numberwithin{equation}{section}

\renewcommand\Re{\operatorname{Re}}

\newcommand{\cA}{\mathcal{A}}

\newcommand{\cD}{\mathcal{D}}
\newcommand{\cF}{\mathcal{F}}

\newcommand{\cO}{\mathcal{O}}
\newcommand{\cP}{\mathcal{P}}

\newcommand{\cT}{\mathcal{T}}

\newcommand{\cI}{\mathcal{I}}
\newcommand{\cL}{\mathcal{L}}

\newcommand{\IC}{\mathbb{C}}

\newcommand{\IN}{\mathbb{N}}
\newcommand{\IP}{\mathbb{P}}
\newcommand{\IR}{\mathbb{R}}

\newcommand{\bA}{\bm{A}}

\newcommand{\bbf}{\bm{f}}

\newcommand{\bu}{\bm{u}}
\newcommand{\bv}{\bm{v}}

\renewcommand{\d}{{\text{d}}}
\newcommand{\pr}[3]{\left(#2, #3\right)_{#1}}
\newcommand{\dpr}[3]{\left\langle #2, #3\right\rangle_{#1'\times #1}}
\newcommand{\dprs}[3]{\left\langle #2, #3\right\rangle_{#1}}
\newcommand{\ceil}[1]{\lceil #1\rceil}
\newcommand{\floor}[1]{\lfloor #1\rfloor}

\newcommand{\norm}[2]{{\Vert #2 \Vert}_{#1}}
\newcommand{\snorm}[2]{{\vert #2 \vert}_{#1}}
\newcommand{\abs}[1]{{\left\vert #1 \right\vert}}

\newcommand{\seminorm}[2]{\left\vert #2 \right\vert_{#1}}
\newcommand{\tripnorm}[2]{\left\lvert\hspace{-1 pt}\left\lvert\hspace{-1 pt}\left\lvert {#2}\right\lvert\hspace{-1 pt}\right\lvert\hspace{-1 pt}\right\lvert_{#1}}

\DeclareMathOperator{\spann}{span}

\newcommand\range[1]{{\varrho(#1)}}
\newcommand\average[1]{{\mu(#1)}}

\newcommand{\lp}[2]{L_{#1}\left(#2\right)}
\newcommand{\hs}[2]{H^{#1}\left(#2\right)}
\newcommand{\hso}[2]{H_0^{#1}\left(#2\right)}

\newcommand{\hst}[2]{\widetilde{H}^{#1}\left(#2\right)}
\newcommand{\hstl}[2]{\widetilde{H}_L^{#1}\left(#2\right)}
\newcommand{\hstr}[2]{\widetilde{H}_R^{#1}\left(#2\right)}
\newcommand{\cont}[2]{C^{#1}\left(#2\right)}
\newcommand{\contl}[2]{C^{#1}_{L}\left(#2\right)}
\newcommand{\contr}[2]{C^{#1}_{R}\left(#2\right)}
\newcommand{\conto}[2]{C_0^{#1}\left(#2\right)}

\newcommand{\flder}[2]{{}_{0}\cD_x^{#1}{#2}}
\newcommand{\hflder}[2]{{}_{0}\hat\cD_x^{#1}{#2}}
\newcommand{\frder}[2]{{}_{x}\cD_1^{#1}{#2}}

\newcommand{\flderd}[3]{\cD_{#1}^{#2}\, {#3}}

\newcommand{\flint}[2]{{}_{0}\cI_x^{#1}{#2}}
\newcommand{\frint}[2]{{}_{x}\cI_1^{#1}{#2}}
\newcommand{\sh}{\frac{s}{2}}

\newcommand{\bil}[2]{a\!\left( {#1}, {#2}\right)}

\newcommand{\lin}[1]{F\!\left({#1}\right)}

\DeclareMathOperator*{\essinf}{ess\,inf}

\usepackage{a4wide}
\title{Fractional differential equations: non-constant coefficients, simulation and model reduction}
\author{Ruben Aylwin}
\author{G\"oksu Oruc}
\author{Karsten Urban}
\affil{Ulm University, Institute of Numerical Mathematics, Helmholtz\-str.\ 20, 89081 Ulm, Germany}
\begin{document}

% \author{Ruben Aylwin}
% \address{Ulm University, Institute of Numerical Mathematics, Helmholtz\-str.\ 20, 89081 Ulm, Germany}
% \curraddr{}
% \email{ruben.aylwin-pincheira@uni-ulm.de}

% \author{G\"{o}ksu Oruc}
% \address{Ulm University, Institute of Numerical Mathematics, Helmholtz\-str.\ 20, 89081 Ulm, Germany}
% \curraddr{}
% \email{goeksu.oruc@uni-ulm.de}

% \author{Karsten Urban}
% \address{Ulm University, Institute of Numerical Mathematics, Helmholtz\-str.\ 20, 89081 Ulm, Germany}
% \curraddr{}
% \email{karsten.urban@uni-ulm.de}

% \thanks{The work on this paper has been funded by the \emph{Federal Ministry for Economic Affairs and Energy of Germany} (BMWE -- Bundesministerium für Wirtschaft und Energie der Bundesrepublik Deutschland).}

% \subjclass{
%   Primary 26A33,% Fractional derivatives and integrals
%   34K37,%Functional-differential equations with fractional derivatives
%   65N30%Finite element, Rayleigh-Ritz and Galerkin methods for boundary value problems involving PDEs
% }

% \keywords{Fractional boundary value problem, variational formulation, Riemann-Liouville fractional derivatives, model order reduction, greedy algorithm, error estimation, Kolmogorov $n$-width}

\date{\today}
%-------------------------------------------------------------------------------------------------------

\maketitle
\begin{abstract}
   We consider boundary value problems with Riemann-Liouville fractional derivatives of order $s \in (1,2)$ with non-constant diffusion and reaction coefficients. A variational formulation is derived and analyzed leading to the well-posedness of the continuous problem and its Finite Element discretization. Then, the Reduced Basis Method through a greedy algorithm for parametric diffusion and reaction coefficients is analyzed. Its convergence properties, and in particular the decay of the Kolmogorov $n$-width, are seen to depend on the fractional order $s$. Finally, numerical results confirming our findings are presented.
 \end{abstract}

%-------------------------------------------------------------------------------------------------------
\section{Introduction}
\label{sec:intro}
%-------------------------------------------------------------------------------------------------------
Differential equations with fractional derivatives have been widely studied in the literature, which is also due to the fact that those type of equations model phenomena, which are relevant in various fields. In this paper, we consider fractional order source problems with non-constant coefficients of the form
\begin{subequations}
\label{eq:PDE}
\begin{align}
	-\flder{\sh}{\left(d(x)\, \flder{\sh}{u}\right)}  + r(x)\, u(x)
		&=f(x), \qquad x\in\Omega:=(0,1),  
		\label{eq:problem} \\
	u(0)=0, \, u(1)&=0, \label{eq:bv}
\end{align}
\end{subequations}
where $u:\Omega\to\IR$ denotes the unknown function, $\flder{\beta}{}$ is the left-sided Riemann-Liouville fractional derivative of order
$\beta > 0$, $s \in (1, 2)$ is the order of the fractional differential equation, $r:\Omega\to\IR$ is the reaction coefficient,
$d: \Omega\to\IR$ is the diffusion coefficient and $f: \Omega\to\IR$ is the right-hand side. 

Systems like \cref{eq:PDE}  have been studied in \cite{jin} for the case of constant diffusion, i.e., $d\equiv 1$. We are interested in the more general case also since the fractional operator in \cref{eq:problem} appears, for example, when considering the fractional Fick's law of diffusion \cite{schumer2001eulerian} together with fractional conservation of mass \cite{wheatcraft2008fractional}. Different operators involving non-constant coefficients have been considered also in \cite{mao2016efficient, wang2013wellposedness}, while applications of fractional derivatives include sub-difussive \cite{jin2019subdiffusion,mustapha2018fem} and super-diffusive \cite{li2018time} processes, which also motivate the present article.

One source of motivation for this paper is model reduction for problems of fractional order. Since those kind of problems are non-local, model reduction might offer additional potential for reduction. To this end, we consider a parameterized version of \cref{eq:PDE}, where $d$, $r$ and the right-hand side $f$ may vary depending on values of parameters $\mu\in\cP$, where $\cP\subset\IR^P$. One might think of a material with different diffusion coefficients in different areas. The Reduced Basis Method (RBM) is a well-established model order reduction technique for parameterized partial differential equations (PPDEs), see e.g.\ \cite{bonito2020reduced,Haasdonk:RB,Rozza:RB,Quarteroni:RB,Urban:RB}. We aim at extending the RBM to fractional-type problems. The RBM relies on a well-posed variational formulation of the parameterized problem. Hence, we aim at generalizing the results for constant coefficients in \cite{jin} to the non-constant coefficient case. 

The next ingredient for the RBM is a \enquote{truth} solver which is able to determine the solution of the problem for a given value of the parameter up to any desired accuracy. Those detailed solutions (also called \enquote{snapshots}) are used in an offline training phase to derive a reduced model by maximizing an error estimator w.r.t.\ a finite training subset $\cP_{\text{train}}\subset\cP$ of parameters in a greedy manner. Hence, an efficiently computable and sharp error estimator is needed, which can be formed by the inverse of the coercivity constant multiplied by the dual norm of the residual. This is the reason why we derive a formula for the coercivity constant as this also makes the dependency of the fractional order $s$ explicit. 

The best possible error that can be achieved by the RBM of size $n\in\IN$ is given be the \emph{Kolmogorov $n$-width} $d_n(\cP)$. It is known that $d_n(\cP)$ decays exponentially fast for elliptic and (space-time variational) parabolic problems, but may show very poor decay for transport and wave-type phenomena, see e.g.\ \cite{AGU25,m.ohlbergers.rave2016}. Hence, we are interested how $d_n(\cP)$ behaves w.r.t.\ the fractional order $s$ as $s\to 2$ is expected to show elliptic and $s\to 1$ transport-type behavior. 

The remainder of this paper is organized as follows. In \Cref{sec:fpde} we collect some notation and known facts on the Riemann-Liouville fractional operators. We also derive a norm equivalence in \Cref{lem:eqNorm}, which is crucial for the subsequent analysis. \Cref{ssec:varfor} is devoted to the derivation and the analysis of a variational formulation for fractional differential equations with non-constant coefficients, as well as their solutions, extending the results in \cite{jin}. A Finite Element discretization is presented and analyzed in \Cref{sec:fem}. We describe the application of the RBM for parameterized diffusion and reaction coefficients in \Cref{sec:rb}. Whereas the application of the RBM turns out to be rather standard (since a coercive variational formulation has been derived in \Cref{ssec:varfor} before), the analysis of the decay of the Kolmogorov $n$-width w.r.t.\ the fractional order $s$ in \Cref{ssec:Kolnwith} is (to the very best of our knowledge) new. Numerical experiments both for the FEM and the RBM are presented in \Cref{sec:numexp}. We close with some conclusions and an outlook in \Cref{sec:conclusion}.

%-------------------------------------------------------------------------------------------------------
\section{Some facts on Fractional Differential Operators}
\label{sec:fpde}
%-------------------------------------------------------------------------------------------------------
\subsection{Notation}
%-------------------------------------------------------------------------------------------------------
% Introduce the classic Sobolev Spaces H^s \tilde{H}^s, only what you can find in books.
% General notation
Let $U$ be a Banach space with norm $\norm{U}{\cdot}$ and dual space $U'$. The space of bounded linear operators between two Banach spaces $U$ and $V$ is written as $\cL(U, V)$. For a Hilbert space $H$, we let $\pr{H}{\cdot}{\cdot}$ and $\dpr{H}{\cdot}{\cdot}$ denote its inner and duality products, respectively. Since we will be dealing mostly with spaces of real-valued functions, all inner and duality products are understood in the bilinear sense. 

For $m\in\IN$ and an open and bounded Lipschitz domain $\Omega\in\IR^d$, $d\in\{1, 2, 3\}$, the
space of real valued continuous functions on $\Omega$ with $m$ continuous derivatives in $\Omega$ is denoted by $\cont{m}{\Omega}$, $\cont{\infty}{\Omega}$ is the space of infinitely continuously differentiable functions in $\Omega$ and $\conto{\infty}{\Omega}$ is the space of functions in $\cont{\infty}{\Omega}$ with compact support in $\Omega$. Furthermore, for $p\geq 1$, the usual space of $p$-integrable measurable functions over $\Omega$ is denoted by $\lp{p}{\Omega}$ and, for $s\in\IR$, we use the Sobolev spaces (of broken order) $\hs{s}{\Omega}$, $\hso{s}{\Omega}$ and $\hst{s}{\Omega}$ as in \cite{mclean2000strongly} and recall the duality relationships 
\begin{align*}
	\hs{s}{\Omega}' = \hst{-s}{\Omega} 
	\quad\text{and}\quad 
	\hst{s}{\Omega}' = \hs{-s}{\Omega}, 
\end{align*}
as well as the equivalence
\begin{align*}
	\hst{s}{\Omega} = \hso{s}{\Omega},\quad
	\forall s\geq 0 \text{ such that }s\not\in\{\tfrac{1}{2},\tfrac{3}{2},\hdots\},
\end{align*}
with equivalent norms (\emph{c.f.}~\cite[Thm.\ 3.33]{mclean2000strongly}). The semi-norm of the Sobolev space of order $s> 0$ is written as $\snorm{\hs{s}{\Omega}}{\cdot}$. We will always identify the space $\lp{2}{\Omega}$ with its dual, so as to obtain Gelfand triples $\hst{-s}{\Omega}\hookrightarrow\lp{2}{\Omega}\hookrightarrow\hs{s}{\Omega}$ and $\hs{-s}{\Omega}\hookrightarrow\lp{2}{\Omega}\hookrightarrow\hst{s}{\Omega}$. Furthermore, we will be required to work with Sobolev spaces of functions over $\Omega$ that have smooth extensions by zero to the left and right sides of $\Omega$, defined as
\begin{align*}
  \hstl{s}{\Omega}&:=\{u\in\hs{s}{\Omega}:\, \exists\, \tilde{u}\in\hst{s}{0, \infty} \text{ s.t.\ } u = \tilde{u}\vert_{\Omega} \},\\
  \hstr{s}{\Omega}&:=\{u\in\hs{s}{\Omega}:\, \exists\, \tilde{u}\in\hst{s}{-\infty, 1} \text{ s.t.\ } u = \tilde{u}\vert_{\Omega} \}.
\end{align*}
A Hilbert space structure is recovered for $\hstl{s}{\Omega}$ by restricting the norm and product of $\hs{s}{-\infty, 1}$ to left-sided extension by zero of elements in $\hstl{s}{\Omega}$. Analogously, restricting the norm and product of $\hs{s}{0, \infty}$ to the right-sided extension by zero of elements in $\hstr{s}{\Omega}$ yields the corresponding Hilbert space structure for this space. These spaces were introduced in \cite[\S 2]{jin} and characterize the range and domain of fractional integral and differential operators to be defined later on. Furthermore, consider the spaces
\begin{align*}
  \contl{\infty}{\Omega} &:=\{u = \tilde{u}|_{\Omega}\text{ for some }\tilde{u}\in\conto{\infty}{0, \infty}\},\\
  \contr{\infty}{\Omega} &:=\{u = \tilde{u}|_{\Omega}\text{ for some }\tilde{u}\in\conto{\infty}{-\infty, 1}\}.
\end{align*}
By proceeding analogously as in \cite[Chap.\ 3.6]{mclean2000strongly}, one can see that these spaces are dense
in $\hstl{s}{\Omega}$ and $\hstr{s}{\Omega}$, respectively. Furthermore, we shall identify elements of
$\hst{s}{\Omega}$, $\hstl{s}{\Omega}$ and $\hstr{s}{\Omega}$ with their corresponding extensions by $0$
so as to not introduce additional notation.

Finally, $\cF$ will denote the Fourier transform, $\Gamma$ will correspond to the Gamma function, and $\ceil{s}$ and $\floor{s}$ will denote, for any $s\in\IR$, the smallest integer larger than $s$ and the largest integer smaller than $s$, respectively.

%-------------------------------------------------------------------------------------------------------
\subsection{Fractional Operators}
\subsubsection{Fractional Integral Operators}
% Definition and main properties of Riemann Liouville Diferential operator
We begin by introducing the Rie\-mann-Liouville fractional integral operators of order $s > 0$, which are needed for the construction of fractional differential operators (\emph{c.f.}~\cite[Chap.\ 1.2.3]{samko1993fractional} and \cite[\S 2.1]{kilbas2006theory}).
\begin{definition}[Fractional Integral Operators]\label{defi:RLInt}
	For any $s > 0$, $\varphi\in\lp{1}{0,\infty}$ and $\psi\in\lp{1}{-\infty,1}$, we introduce the \emph{left-} and \emph{right-sided Riemann-Liouville fractional integral operators} as
	\begin{align*}
		&\flint{s}{\varphi}(x):= \frac{1}{\Gamma(s)}\int_0^{x}(x - t)^{s - 1}\varphi(t)\d t\quad\forall x > 0, \\
		&\frint{s}{\psi}(x):= \frac{1}{\Gamma(s)}\int_x^{1}(t - x)^{s - 1}\psi(t)\d t \quad\forall x < 1.
	\end{align*}
\end{definition}

The following lemmas introduce useful properties of the integral operators, which have appeared in \cite{jin,kilbas2006theory,samko1993fractional}. We omit the proofs for brevity, but point out relevant references for each result. We also refer to \cite[La.\ 2.1 and 2.3]{kilbas2006theory} for more details.

\begin{lemma}[{\cite[Thm.\ 2.6]{samko1993fractional}}]\label{lem:IntSg}
	The fractional integral operators introduced in \Cref{defi:RLInt} are bounded in $\lp{p}{\Omega}$, $\Omega=(0,1)$,  for any $p\geq 1$, with continuity constant $1/\Gamma{(s+1)}$. Moreover, for any $s_1, s_2 > 0$ and $\varphi\in\lp{p}{\Omega}$ they satisfy $\flint{s_1+s_2}{\varphi} = \flint{s_2}{\flint{s_1}{\varphi}}$ and $\frint{s_1+s_2}{\varphi} = \frint{s_2}{\frint{s_1}{\varphi}}$.\hfill$\Box$
\end{lemma}

\begin{lemma}[{\cite[Thm.\ 3.1 and Rem.\ 3.1]{jin}}]\label{lem:IntCont}
	Let $s,\sigma \geq 0$. The fractional integral operators $\flint{s}{}$ and $\frint{s}{}$ are bounded from $\hstl{\sigma}{\Omega}$ to $\hstl{\sigma + s}{\Omega}$ and from $\hstr{\sigma}{\Omega}$ to $\hstr{\sigma + s}{\Omega}$, respectively.\hfill$\Box$
\end{lemma}

\begin{lemma}[{\cite[Cor.\ to Thm.\ 3.5]{samko1993fractional}}]\label{lem:intByInt}
  Let $s\geq 0$. For any $\psi, \phi\in\lp{2}{\Omega}$ the following relation holds $\pr{\lp{2}{\Omega}}{\flint{s}{\psi}}{\phi} = \pr{\lp{2}{\Omega}}{\psi}{\frint{s}{\phi}}$.\hfill$\Box$
\end{lemma}

%---------------------------------------------------------------------
\subsubsection{Fractional Differential Operators}
We are now ready to introduce the fractional Riemann-Liouville derivatives and their properties.

\begin{definition}[Fractional Derivatives]\label{defi:RLDif}
	For any non-integer $s > 0$, $\varphi\in\contl{\infty}{\Omega}$ and $\psi\in\contr{\infty}{\IR}$, we introduce the \emph{left-sided} $\flder{s}{}$ and \emph{right-sided} $\frder{s}{}$ \emph{Rie\-mann-Liouville fractional derivative} for all $x\in \Omega$ as
	\begin{align*}
		\flder{s}{\varphi}(x) &:= \frac{\d^{\ceil{s}}}{\d x^{\ceil{s}}}\left(\flint{\ceil{s} - s}{\varphi}(x)\right)
		\quad\text{and}\quad
		\frder{s}{\psi}(x):= -\frac{\d^{\ceil{s}}}{\d x^{\ceil{s}}}\left(\frint{\ceil{s} - s}{\psi}(x)\right).
	\end{align*}
\end{definition}
Identifying $\varphi\in\hst{s}{\Omega}$ with its zero extension to $\IR$ allows us to extend the fractional derivatives to the outside of $\Omega$ as well.

\begin{lemma}[{\cite[Thm.\ 2.2]{jin}}]\label{lem:derCont}
	For any $s > 0$ the fractional differential operators $\flder{s}{}$ and $\frder{s}{}$ have bounded extensions from $\hstl{s}{\Omega}$ and $\hstr{s}{\Omega}$ to $\lp{2}{\Omega}$, respectively. Moreover, $\norm{\lp{2}{\Omega}}{\flder{s}{\varphi}}\leq\norm{\hstl{s}{\Omega}}{\varphi}$ for all $\varphi\in\hstl{s}{\Omega}$ and $\norm{\lp{2}{\Omega}}{\frder{s}{\psi}}\leq\norm{\hstr{s}{\Omega}}{\psi}$ for all $\psi\in\hstr{s}{\Omega}$.\hfill$\Box$
\end{lemma}

\begin{lemma}[{\cite[Prop.\ A.4]{ervin2006variational} or \cite[Thm.\ 2.4]{samko1993fractional}}]\label{lem:derInv}
	  The left and right-sided fractional derivatives of order $s>0$ act as left inverses of the left- and right-sided fractional integral operators of order $s$, i.e., $\flder{s}{\flint{s}{\varphi}}(x) = \varphi(x)$ and $\frder{s}{\frint{s}{\varphi}}(x) = \varphi(x)$,  whenever $\varphi(x)$ is a summable function.\hfill$\Box$
\end{lemma}

\begin{lemma}[{\cite[La.\ 4.1]{jin}}]\label{lem:exchange}
	Let $s\in (0,1)$. Then, for all $x\in\Omega$ it holds  $\flder{s}{\varphi}(x) = \flint{1-s} \varphi'(x)$ for all $\varphi\in \contl{\infty}{\Omega}$ and $\frder{s}{\psi}(x) = -\frint{1-s} \psi'(x)$ for $\psi\in \contr{\infty}{\Omega}$.  Furthermore, these relationships can be extended to hold for $\varphi\in\hstl{1}{\Omega}$ and $\psi\in\hstr{1}{\Omega}$, respectively.\hfill$\Box$
\end{lemma}

The next statement is a generalization of integration by parts and is fundamental for the subsequent derivation of a variational formulation.
\begin{proposition}[{\cite[\S 4.1]{jin}}]\label{prop:intByParts}
	Let $\varphi\in\contl{\infty}{\Omega}$, $\psi\in\contr{\infty}{\Omega}$ and $s\in (0,1)$. Then,   
	$\pr{\lp{2}{\Omega}}{\flder{s}{\varphi}(x)}{\psi(x)} =  \pr{\lp{2}{\Omega}}{\varphi(x)}{\frder{s}{\psi}(x)}$.
\end{proposition}
%----------------------------------------------------------------------------------------------
\subsection{Equivalent Norms in $\hst{s}{\Omega}$}
We continue by analyzing equivalent norms in $\hst{s}{\Omega}$ that will facilitate our analysis of variational formulations of FPDEs. 

\begin{lemma}\label{lem:eqNorm}
  Let $s>0$, $s\not\in\{\tfrac{1}{2},\tfrac{3}{2},\hdots\}$. Then, for $\varphi\in\hst{s}{\Omega}$ we have the relation 
	$    
	\pr{\lp{2}{\Omega}}{\flder{s}{\varphi}}{\frder{s}{\varphi}} 
	= \cos\left(\pi s\right)\norm{\lp{2}{\IR}}{\flder{s}{\varphi}}^2 
	= \cos\left(\pi s\right)\norm{\lp{2}{\IR}}{\frder{s}{\varphi}}^2$.
\end{lemma}
\begin{proof}
  The proof follows from \cite[Thm.\ 2.3, Lem.\ 2.4]{ervin2006variational}.
\end{proof}

\begin{proposition}\label{Cor:NormEq}
	Set $\tripnorm{s}{\varphi}:=(\norm{\lp{2}{\Omega}}{\varphi}^2+\norm{\lp{2}{\IR}}{\flder{s}{\varphi}}^2)^{\frac{1}{2}}$ as well as $\seminorm{s}{\varphi}:=\norm{\lp{2}{\IR}}{\flder{s}{\varphi}}$. Then, for all $\varphi\in \hst{s}{\Omega}$,
	\begin{align}
		\norm{\hst{s}{\Omega}}{\varphi}
		\leq \tripnorm{s}{\varphi}
		\leq \sqrt{2}\, \norm{\hst{s}{\Omega}}{\varphi},
		&\qquad
		\tfrac{\Gamma(s+1)}{\sqrt{2}} \norm{\hst{s}{\Omega}}{\varphi}
		\le \seminorm{s}{\varphi}
		\leq \norm{\hst{s}{\Omega}}{\varphi}
		\label{eq:NE}.
	\end{align}
\end{proposition}
\begin{proof}
  By consequence of \cite[Thm.\  2.10]{ervin2006variational} together with \Cref{lem:IntSg}
  we have that	$\norm{\lp{2}{\Omega}}{\varphi} \leq {\Gamma\left(s+1\right)^{-1}}\norm{\lp{2}{\Omega}}{\flder{s}{\varphi}}$
  for all $\varphi\in\hstl{s}{\Omega}$.
  Further considering $\varphi\in\conto{\infty}{\Omega}$, we have that
  $$2\, \norm{\lp{2}{\IR}}{\flder{s}{\varphi}}^2
  \geq \Gamma\left(s+1\right)^2\norm{\lp{2}{\Omega}}{\varphi}^2 + \norm{\lp{2}{\IR}}{\flder{s}{\varphi}}^2
  \geq \Gamma\left(s+1\right)^2 \tripnorm{s}{\varphi}^2,$$
  where we have used that $\Gamma(x) < 1$ for $x\in(1,2)$. Next, we use
  \begin{align}\label{eq:Lemma2.7}
    \cF(\flder{s}{\varphi})(\omega)=(-\imath\omega)^{s}\cF(\varphi)(\omega)
    \quad\text{for}\quad \varphi\in\conto{\infty}{\IR}
  \end{align}
  (see, e.g.\ \cite[Rem.\ 2.11]{kilbas2006theory}) to deduce 
  \begin{align*}
    \tripnorm{s}{\varphi}^2
    &= \int\limits_\IR (1+\abs{\omega}^{2s})\abs{\cF(\varphi)(\omega)}^2\d\omega
      \geq \int\limits_\IR (1+\abs{\omega}^{2})^s\abs{\cF(\varphi)(\omega)}^2\d\omega 
      = \norm{\hs{s}{\IR}}{\varphi}^2.
  \end{align*}
  In addition, we have that $\norm{\hs{s}{\IR}} \varphi \leq\sqrt{2}\,\Gamma(s+1)^{-1}\norm{\lp{2}{\IR}}{\flder{s}{\varphi}}$, i.e., the left-handed inequalities in \eqref{eq:NE} for all $\varphi\in\conto{\infty}{\Omega}$. Since $\conto{\infty}{\Omega}$ is dense in $\hst{s}{\Omega}$, we can conclude the estimates also for $\varphi\in\hst{s}{\Omega}$. 
  Concerning the upper bounds, consider $\varphi\in\contl{\infty}{\Omega}$ and, by the Plancherel's theorem and \eqref{eq:Lemma2.7} we get that
  \begin{align*}
    \norm{\lp{2}{\Omega}}{\flder{s}{\varphi}}
    &\leq \seminorm{s}{\varphi}
      =  \norm{\lp{2}{\IR}}{\flder{s}{\varphi}} 
      = \norm{\lp{2}{\IR}}{\cF(\flder{s}{\varphi})(\omega)}
      = \norm{\lp{2}{\IR}}{(-\imath\omega)^s\cF({\varphi})(\omega)}\\ 
    &= \int\limits_\IR \abs{\omega}^{2s}\abs{\cF(\varphi)(\omega)}^2\d\omega
      \leq \int\limits_\IR(1+\abs{\omega}^2)^s\abs{\cF(\varphi)(\omega)}^2\d\omega = \norm{\hs{s}{\IR}}{\varphi}.
  \end{align*}
  The remaining claims follow as in the proofs of \cite[Thm.\ 2.1 and 2.2]{jin}.
\end{proof}

%-------------------------------------------------------------------------------------------------------
\section{Riemann-Liouville Fractional Problem}
\label{ssec:varfor}
% -------------------------------------------------------------------------------------------------------
% -------------------------------------------------------------------------------------------------------
We continue by deriving a variational formulation for \eqref{eq:PDE} (which, to the best of our knowledge, has not
been considered before) and analyze its well-posedness and smoothness of its solutions.
\subsection{Variational Formulation}
We shall work under the following assumptions on the data:
\begin{assumption}\label{as:param}
  Let $s\in (1, 2)$, $d\in \lp{\infty}{\Omega}$ such that $d(x)\ge d_0>0$ for almost all $x\in\Omega$,  $r\in\lp{\infty}{\Omega}$ and $f\in\lp{2}{\Omega}$.
\end{assumption}

Then, we shall consider the fractional differential operator with non-constant diffusion coefficients as
\begin{align}\label{eq:secondOrderDOp}
	\flderd{d}{s}{u}(x) :=  \flder{\sh}{\left(d(x)\, \flder{\sh}{u}(x)\right)},
\end{align}
As mentioned in \Cref{sec:intro}, similar operators were studied in \cite{mao2016efficient}--where the innermost left-sided derivative in \cref{eq:problem} is replaced by a right-sided derivative--and in \cite{wang2013wellposedness}--where the coefficent stands \enquote{outside} of the fractional derivative, i.e., $\tfrac{\d}{\d x}(d(x)\flder{s-1}{u}(x))$. Both choices lead to milder conditions on the parameter $d(x)$ for the well-posedness of the considered variational formulations.

We begin by considering the left-hand side of \cref{eq:problem} with $u\in \conto{\infty}{\Omega}$. Multiplication with a test function $v\in \conto{\infty}{\Omega}$ and integration over $\Omega$ yields
\begin{align} \label{eq:varForm1}
  	-\pr{\lp{2}{\Omega}}{\flderd{d}{s}{u}}{v} + \pr{\lp{2}{\Omega}}{r\, u}{v}.
\end{align}
Applying \Cref{prop:intByParts} yields,
\begin{align*}
 	 -\pr{\lp{2}{\Omega}}{\flderd{d}{s}{u}}{v} 
	 	&=   -\pr{\lp{2}{\Omega}}{\flder{\sh} {\big(} d\, {\flder{\sh}{u}} {\big)} }{v}
		= -\pr{\lp{2}{\Omega}}{d\, \flder{\sh}{u}}{\frder{\sh}{v}}.
\end{align*}
Therefore, given \Cref{as:param}, we introduce the following bilinear and linear forms
\begin{align}
	\bil{u}{v}
		&:=-\pr{\lp{2}{\Omega}}{ d\, \flder{\sh}{u}}{\, \frder{\sh}v} 
			+ \pr{\lp{2}{\Omega}}{r\, u}{v}
		=: a_1(u,v) + a_2(u,v),
		\label{eq:nonConstBilForm}\\
  	\lin{v}
		&:=\pr{\lp{2}{\Omega}}{f}{v}.\label{eq:linForm}
\end{align}

\begin{problem}\label{prob:nonConstCoeff}
  	Seek $u\in\hst{\sh}{\Omega}$ such that $\bil{u}{v} = \lin{v}$ for all $v\in\hst{\sh}{\Omega}$.
\end{problem}

To investigate the well-posedness of \Cref{prob:nonConstCoeff}, we need some preparations.

\begin{lemma}\label{lem:nonConstCont}
  If \Cref{as:param} holds, the bilinear form in \cref{eq:nonConstBilForm} is continuous
  in $\hst{\sh}{\Omega}$, i.e., with $C_{d,r}:= 2\, (\norm{\lp{\infty}{\Omega}}{d} + \norm{\lp{\infty}{\Omega}}{r})$
  \begin{align*}
  	\bil{u}{v} \le  C_{d,r}\,
		\norm{\hst{s/2}{\Omega}}{u} \norm{\hst{s/2}{\Omega}}{v}
		\quad \forall\, u, v\in \hst{\sh}{\Omega}.
  \end{align*}
\end{lemma}
\begin{proof}
  Clearly, for $u,v\in\hst{\sh}{\Omega}$, it holds that
  \begin{align*}
    a_1(u,v)
    &\leq
    \norm{\lp{\infty}{\Omega}}{d}\, \abs{\pr{\lp{2}{\Omega}}{\flder{\sh}{u}}{\, \frder{\sh}v}} 
    \leq \norm{\lp{\infty}{\Omega}}{d}\,\norm{\lp{2}{\Omega}}{\flder{\sh}{u}}\norm{\lp{2}{\Omega}}{\frder{\sh}{v}}\\
    &\leq \norm{\lp{\infty}{\Omega}}{d}\,\norm{\lp{2}{\IR}}{\flder{\sh}{u}}\norm{\lp{2}{\IR}}{\frder{\sh}{v}}   
    = \norm{\lp{\infty}{\Omega}}{d}\,\norm{\lp{2}{\IR}}{\flder{\sh}{u}}\norm{\lp{2}{\IR}}{\flder{\sh}{v}} 
  \end{align*}
  yielding the bound  
  	\begin{align*}
		\bil{u}{v} 
		&\le (\norm{\lp{\infty}{\Omega}}{d} + \norm{\lp{\infty}{\Omega}}{r})\,
		(\norm{\lp{2}{\IR}}{\flder{\sh}{u}}^2 + \norm{\lp{2}{\Omega}}{u}^2)^{1/2}
		\\
		&\qquad\qquad\qquad\qquad\qquad\quad\times
		(\norm{\lp{2}{\IR}}{\flder{\sh}{v}}^2 + \norm{\lp{2}{\Omega}}{v}^2)^{1/2}\\
		&\le 2\, (\norm{\lp{\infty}{\Omega}}{d} + \norm{\lp{\infty}{\Omega}}{r})\,
		\norm{\hst{s/2}{\Omega}}{u} \norm{\hst{s/2}{\Omega}}{v},
	\end{align*}
  	by \Cref{Cor:NormEq}, which proves the claim.
\end{proof}

To prove the coercivity of the bilinear form in \cref{eq:nonConstBilForm}, we need additional conditions on $d\in\lp{\infty}{\Omega}$.

\begin{assumption}\label{as:dParam}
	We define the \emph{average} $\average{d}$ and the \emph{range} $\range{d}$  of $d$ as
	\begin{align*}
	\average{d}
		:=\frac{1}{2}\Big(\esssup\limits_{x\in\Omega}d(x) 
			+ \essinf\limits_{x\in\Omega}d(x)\Big),
	&\quad
	\range{d}
		:=\frac{1}{2}\Big(\esssup\limits_{x\in\Omega}d(x) 
			- \essinf\limits_{x\in\Omega}d(x)\Big)
	\end{align*}
	as well as $\underline{r} := \essinf\limits_{x\in\Omega}r(x)$. 
	Setting $\gamma_{s,d}:= \average{d}\,\abs{\cos\left(s\tfrac{\pi}{2}\right)} - \range{d}$, we assume in addition to \Cref{as:param} that 
  	\begin{align}\label{eq:Ass-d}
		c_{s,d,r} := \gamma_{s,d} \tfrac{\Gamma(s/2+1)^2}{4} + \underline{r} \ge 0.
	\end{align}
\end{assumption}

\begin{remark}
\begin{compactenum}[(a)]
	\item It is immediate that $\norm{\lp{\infty}{\Omega}}{d-\average{d}} = \range{d}$.
	\item As $s\to 2$, the value $\abs{\cos(s\tfrac{\pi}{2})}$ converges to $1$, so that $\gamma_{s,d}=\essinf\limits_{x\in\Omega}d(x)$. On the other hand, if $s\to 1$, the value of  $\cos(s\tfrac{\pi}{2})$ tends to zero, which means that $d$ can only vary very little since $\range{d}$ must be very small to satisfy \eqref{eq:Ass-d}.
\end{compactenum}
\end{remark}

The following result is a generalization of \cite[Thm.\ 4.3]{jin} for non-constant diffusion. Moreover, we shall make the involved constants explicit for later use.

\begin{theorem}\label{thm:uniqueNonConst}
  If \Cref{as:param,as:dParam} hold, $a(\cdot,\cdot)$ is coercive, i.e.,
  \begin{align*}
  	a(u,u) \ge \alpha_{s,d}\, \norm{\hst{\sh}{\Omega}}{u}^2,\,\, u\in  \hst{\sh}{\Omega}
	\qquad\text{where }
	\alpha_{s,d} = \gamma_{s,d} \tfrac{\Gamma(s/2+1)^4}{8}.
  \end{align*}
  Moreover, \Cref{prob:nonConstCoeff} admits a unique solution $u\in\hst{\sh}{\Omega}$ such that
  \begin{align}\label{eq:contDepNon}
    \norm{\hst{\sh}{\Omega}}{u}\leq \tfrac{1}{\alpha_{s,d}} \norm{\lp{{2}}{\Omega}}{f}.
  \end{align}
\end{theorem}
\begin{proof}
	  By \Cref{lem:eqNorm}, we have that
	  \begin{align*}
	    a_1(u,u) 
	    &= -{\average{d}\pr{\lp{2}{\Omega}}{\flder{\sh}{u}}{\frder{\sh}{u}}} 
	    	- {\pr{\lp{2}{\Omega}}{(d-\average{d})\,\flder{\sh}{u}}{\frder{\sh}{u}}} \\
	    &\geq \left(\average{d}\abs{\cos\left(s\tfrac{\pi}{2}\right)} - \range{d}\right)\norm{\lp{2}{\IR}}{\flder{\sh}{u}}^2
	    = \gamma_{s,d}\, |u|_{s/2}^2.
	  \end{align*}
	  Thanks to \Cref{as:dParam} and by using \Cref{Cor:NormEq}, we get that 
	  \begin{align*}
	  a(u,u) &\ge \gamma_{s,d}\, \seminorm{s/2}{u}^2 + \underline{r}\, \norm{\lp{2}{\Omega}}{u}^2
        \ge \gamma_{s,d}\, \tfrac{\Gamma(s/2+1)^2}{4} \tripnorm{s/2}{u}^2 + \underline{r}\, \norm{\lp{2}{\Omega}}{u}^2\\
	  &\ge \gamma_{s,d}\, \tfrac{\Gamma(s/2+1)^2}{4} |u|_{s/2}^2
	  \ge \gamma_{s,d}\, \tfrac{\Gamma(s/2+1)^4}{8} \norm{\hst{\sh}{\Omega}}{u}^2.
	  \end{align*}
	  Well-posedness follows from the Lax-Milgram- resp.\ Banach-Ne\v{c}as theorem, \cite[\S4.5]{arendt2023partial}.
    \end{proof}

    \begin{remark}\label{Remark:coercivity}
 The coercivity constant given in \Cref{thm:uniqueNonConst} does not depend on $\underline{r}$ in \Cref{as:dParam}, but might not be optimal depending on the value of $\underline{r}$. If, e.g., $\underline{r}\geq 0$, we see that  $a(u,u) \ge \gamma_{s,d}\tfrac{\Gamma(s/2+1)^2}{2}\norm{\hst{\sh}{\Omega}}{u}^2$.
On the other hand, if $\underline{r}<0$, we may argue as in the proof of \Cref{Cor:NormEq} to obtain the bound $a(u,u) \ge \tfrac{1}{2}(\gamma_{s,d}{\Gamma(s/2+1)^2}+{\underline{r}})\, \norm{\hst{\sh}{\Omega}}{u}^2$, and get an alternative form of the coercivity constant as
	\begin{align}\label{eq:tildealpha}
		\tilde\alpha_{s,d,r} := \gamma_{s,d}\tfrac{\Gamma(s/2+1)^2}{2} + \tfrac12 \min\{ \underline{r},0\},
	\end{align}
	which depends on $\underline{r}$ and might be larger than $\alpha_{s,d}$ in \Cref{thm:uniqueNonConst}.
\end{remark}

\begin{remark}\label{Remark:infsup}
	\Cref{prob:nonConstCoeff} may also admit a unique solution under milder conditions replacing coercivity by an inf-sup condition, using Fredholm's alternative and/or the Peetre-Tartar lemma as in \cite{jin} for constant diffusion coefficients. We restrict ourselves here to the coercive case as it allows for an explicit formula for the coercivity constant and also since injectivity does not need to be assumed as in \cite{jin}. This is beneficial for the RBM in  \Cref{sec:rb} as we avoid the use of the \emph{Successive Constraint Method (SCM)} from \cite{SCM} to compute the inf-sup condition.
\end{remark}

\subsection{Strong Solutions}
\label{ssec:ssol}
%------------------------------------------------------------------------------------------
For our subsequent numerical experiments, we are interested in exact solutions in order to be able to compute approximation errors exactly. To this end, we now consider strong solutions to \Cref{prob:nonConstCoeff}. Once again, we follow the presentation in \cite[\S 3]{jin} (there for $d\equiv 1$) and consider the case $r(x)\equiv 0$. Their regularity will be analyzed later on. We set,
\begin{align*}
	g(x) := \,\flint{\sh}{(d(x)^{-1}\,  \flint{\sh}{f})(x)} =: \cI_d^s f(x),
\end{align*}
so that $\flint{1-\sh}{(d\, \flder{\sh}{g})} = \flint{1-\sh}{\flint{\sh}{f}} = \flint{1}f$, which gives $\flderd{d}{s}g=f$. Next, define
\begin{align}
  \label{eq:singTerm}
	\rho(x) := \, \flint{\sh}{[d(x)^{-1}\, x^{\sh-1}]},
	\qquad
	p(x) := \frac{\rho(x)}{\rho(1)}.
\end{align}
Then, $p(0)=\rho(0)=0$ and $p(1)=1$ as well as
\begin{align*}
  \flderd{d}{s}p 
	&= \tfrac{d}{dx}\! \left( \flint{1-\sh}{[d(x)\, \flder{\sh}{p}]}\right)
	= \tfrac1{\rho(1)} \tfrac{d}{dx} \left( \flint{1-\sh} x^{\sh-1}\right)
	= \tfrac1{\rho(1)} \tfrac{d}{dx} (1) 
	= 0,
\end{align*}
Hence, the strong solution of \Cref{prob:nonConstCoeff} reads
\begin{align}\label{eq:strongSolution}
	u=-g + [(\cI_d^s f)(1)]\, p. 
\end{align}
It is clear that $u\in\hst{\sh}{\Omega}$. In order to show that $u$ is a solution of \Cref{prob:nonConstCoeff}, let $v\in C^\infty_0(\Omega)$.Then, 
\begin{align}
  \label{eq:pzero}
  \begin{aligned}
  \dprs{\hs{-\sh}{\Omega}\times\hst{\sh}{\Omega}}{\flderd{d}{\sh}p}{v}
  &= \pr{\lp{2}{\Omega}}{d(x)\, \flder{\sh}p}{\frder{\sh}v}\\
  &\kern-100pt= \tfrac1{\rho(1)} \pr{\lp{2}{\Omega}}{d(x)\,d(x)^{-1}\, x^{\sh-1}}{\frder{\sh}v}
    = \tfrac1{\rho(1)} \pr{\lp{2}{\Omega}}{x^{\sh-1}}{\frint{1-\sh}{v'}}\\
  &\kern-100pt= \tfrac1{\rho(1)} \pr{\lp{2}{\Omega}}{\flint{1-\sh} x^{\sh-1}}{v'}
    = \tfrac{\Gamma(\sh)}{\rho(1)} \pr{\lp{2}{\Omega}}{1}{v'} = 0,
    \end{aligned}
\end{align}
where we have used \Cref{lem:exchange,lem:intByInt,lem:derInv} together with the fact that $\flint{1-\sh}{x^{\sh-1}}=\Gamma(\sh)$
(see, e.g.\ \cite[Prop.\ 2.1]{kilbas2006theory}). Analogously, we have that
\begin{align}
  \label{eq:geqf}
  \begin{aligned}
    \dprs{\hs{-\sh}{\Omega}\times\hst{\sh}{\Omega}}{\flderd{d}{s}g}{v}
    &= \pr{\lp{2}{\Omega}}{d(x)\, \flder{\sh}g}{\frder{\sh}v}\\
    &\kern-100pt= \pr{\lp{2}{\Omega}}{d(x)\, d(x)^{-1} \flint{\sh}{f}}{\frder{\sh}v}
      = \pr{\lp{2}{\Omega}}{\flint{\sh}{f}}{\frder{\sh}v} = \pr{\lp{2}{\Omega}}{f}{v}.
    \end{aligned}
\end{align}
Therefore, we have that by combination of \cref{eq:strongSolution,eq:pzero,eq:geqf}, we have that
\begin{align*}
 	\dprs{\hs{-\sh}{\Omega}\times\hst{\sh}{\Omega}}{\flderd{d}{s}u}{v} = \pr{\lp{2}{\Omega}}{f}{v},
\end{align*}
which shows that \eqref{eq:strongSolution} in fact solves \Cref{prob:nonConstCoeff}.

\subsection{Regularity}
%-----------------------------------------------------------------
We continue analyzing the regularity of the strong solution given in \eqref{eq:strongSolution}. By \cite[Prop.\ 2.1]{kilbas2006theory}, we have, for any $\beta > 0$,
\begin{align*}
	\flder{\beta}x^{\sh-1} = \frac{\Gamma(\sh)}{\Gamma(\sh-\beta)}x^{\sh - \beta - 1},
\end{align*}
which belongs to $\lp{2}{\Omega}$ whenever $\sh-\beta-1 > -\frac{1}{2}$, i.e., $x^{\sh-1}$ belongs to $\hstl{\sh-\frac{1}{2}-\epsilon}{\Omega}$ for any $\epsilon>0$. Assuming sufficient regularity of $d^{-1}$ (e.g., $d^{-1}\in\cont{1}{\overline{\Omega}}$, see \cite[La.\ 3.2]{ervin2006variational}), we have $d^{-1}(x)\,x^{\sh-1}\in\hstl{\sh-\frac{1}{2}-\epsilon}{\Omega}$ and, by \Cref{lem:IntCont}, we conclude that  $p\in\hstl{s-\frac{1}{2}-\epsilon}{\Omega}$ (which, again, concedes with the findings in \cite{jin} for the case $d\equiv 1$). Assuming, on the other hand, $f\in\lp{2}{\Omega}$ yields $g\in\hstl{s}{\Omega}$. Finally, arguing exactly as in the proof of \cite[Thm.\ 4.4]{jin}, yields the following result.

\begin{theorem}\label{thm:reg}
  Let \Cref{as:param,as:dParam} hold. Then, the solution $u$ to \Cref{prob:nonConstCoeff} belongs
  to $\hst{\beta}{\Omega}$ for any $\beta\in [\sh, s-\frac{1}{2})$.\hfill$\Box$
\end{theorem}

\begin{remark}
We note that the result in \Cref{thm:reg} can \emph{not} be improved by assuming additional smoothness on the data $f$ and $d$, since the regularity of the solution in \cref{eq:strongSolution} is limited by the singular term $p$ in \cref{eq:singTerm}. Moreover, the statement remains valid for the case $r\neq 0$, even though we cannot construct the solution for this case as in \Cref{ssec:ssol}.
\end{remark}
%-------------------------------------------------------------------------------------------------------

%-------------------------------------------------------------------------------------------------------
\section{Finite Element Discretization}
\label{sec:fem}
%-------------------------------------------------------------------------------------------------------
In this section, we describe and analyze a Finite Element Method (FEM) for the numerical construction of the solutions to \Cref{prob:nonConstCoeff} based upon the variational formulation developed in \Cref{ssec:varfor}. 

To this end, we use  a family of standard uniform meshes $\{\cT_h\}_{h>0}$ with mesh-size $h>0$ by dividing $\Omega=(0,1)$ into $N_h\in\IN$ sub-intervals, where $N_h>1$ is a (possibly large) integer. We then define discrete subspaces $V_h \subset \hst{\sh}{\Omega}$ of continuous and piecewise linear functions on the sub-intervals in $\cT_h$. Due to the properties of $\hst{\sh}{\Omega}$, we can conclude that the functions in $V_h$ vanish at the boundaries of the domain $\Omega$. 

\subsection{Error Analysis}
%-------------------------------------------------------------------------------------------------------
We begin our analysis by presenting the approximation properties of the finite element space $V_h$ as given in \cite{jin}.

\begin{lemma}[{\cite[La.\ 5.1]{jin}}]\label{lem:FEMUh}
	Let $\nu \in [\sh,2]$, and the family of meshes $\{\cT_h\}_{h>0}$ be quasi-uniform. If $u \in H^{\nu}(\Omega) \cap \tilde{H}^{\sh}(\Omega)$, then there exists a constant $c>0$
	\begin{align}
	\inf_{v_h \in V_h}\norm{\hst{\sh}{\Omega}}{u-v_h} 
		\leq c\, h^{\nu-\sh}\norm{\hs{\nu}{\Omega}}{u}.
	\end{align}
\end{lemma}

Next, we consider the discrete version of \Cref{prob:nonConstCoeff}.
\begin{problem}\label{prob:nonconstCoeffDis}
	Seek $u_h\in V_h$ such that $\bil{u_h}{v_h} = \lin{v_h}$ holds for all $v_h\in V_h$.
\end{problem}

Well-posedness of \Cref{prob:nonconstCoeffDis} is inherited from \Cref{thm:uniqueNonConst} (provided that \Cref{as:dParam} holds) with the coercivity constant $\alpha_{s,d}$. Moreover, we obtain the following standard a priori error bound.

\begin{theorem}\label{thm:FEMerror}
	Let \Cref{as:dParam} hold, then
	\begin{align*}
		\norm{\hst{\sh}{\Omega}}{u-u_h} \leq C h^{\beta^*}\norm{\lp{2}{\Omega}}{f} 
	\end{align*}
	for any $\beta^* \in [0, \frac{s}{2}-\frac{1}{2})$ and some positive $C>0$.
\end{theorem}
\begin{proof}
	  By \cite[Thm.\ 9.42]{arendt2023partial} we obtain the well-known quasi-optimality, namely  
	  $
		\norm{\hst{\sh}{\Omega}}{u_h-u} 
		\leq \tfrac{C_{d,r}}{\alpha_{s,d}}\inf\limits_{v_h\in V_h}\norm{\hst{\sh}{\Omega}}{u-v_h}
		\leq c\,\tfrac{C_{d,r}}{\alpha_{s,d}}\, h^{\nu-\sh}\,\norm{\hs{\nu}{\Omega}}{u}$  
	  with $C_{d,r}$ from \Cref{lem:nonConstCont} and $c>0$ from \Cref{lem:FEMUh}. Then, \Cref{thm:reg} yields the estimate $\norm{\hst{\sh}{\Omega}}{u-u_h} \leq Ch^{\nu-\sh}\norm{\lp{2}{\Omega}}{f}$ for $\nu\in[\sh,s-\frac{1}{2})$ with $C>0$ being a combination of the previous constants with \Cref{thm:reg}. Considering $\beta^*:=\nu-\sh$ yields the desired result. 
\end{proof}

\begin{remark}\label{rmk:convL2}
	As in \cite[Thm.\ 5.3]{jin}, the consideration of an adjoint problem yields the estimate
	\begin{align*}
		\norm{\lp{2}{\Omega}}{u-u_h} + h^{\beta^*}\norm{\hst{\sh}{\Omega}}{u-u_h}\leq C h^{2\beta^*}\norm{\lp{2}{\Omega}}{f}, 
	\end{align*}
	for any $\beta^* \in [0, \frac{s}{2}-\frac{1}{2})$ and some positive $C>0$. 
\end{remark}

\subsection{The Algebraic System}
%-------------------------------------------------------------------------------------------------------
As usual, a FE discretization is realized in terms of a basis, i.e., $V_h=\spann\{ \varphi_i:\, i=1,...,N_h-1\}$, e.g.\ piecewise linear functions.  Then, \Cref{prob:nonconstCoeffDis} amounts solving a linear system of equations $\bA_h \bu_h = \bbf_h$, where $\bA_h = ( a(\varphi_i,\varphi_j) )_{i,j=1,...,N_h-1}\in\IR^{(N_h-1)\times (N_h-1)}$ is the stiffness matrix, the vector $\bbf_h = ( F(\varphi_j) )_{j=1,...,N_h-1}\in\IR^{N_h-1}$ the right-hand side and $\bu_h\in\IR^{N_h-1}$ is the unknown vector of expansion coefficients of $u_h\in V_h$ in terms of the FE basis.

Due to the non-local nature of the fractional integration operator, it is clear that $\bA_h$ is non-symmetric and densely populated. Without any further compression, the setup of $\bA_h$ and any matrix-vector multiplication e.g.\ used within a GMRES iteration requires $\cO(N_h^2)$ operations. Moreover, the number of GMRES iterations is known to depend on the condition number $\kappa(\bA_h)$, which we are going to analyze next.

To this end, recall the maximum and minimum singular values
\begin{align*}
	\sigma_{h, \text{max}} := \kern-6pt\sup\limits_{\bu_h\in\IR^{N_h-1}}\sup\limits_{\bv_h\in\IR^{N_h-1}} \varrho(\bA_h,\bu_h,\bv_h),
	\,\,
	\sigma_{h, \text{min}} := \kern-6pt\inf\limits_{\bu_h\in\IR^{N_h-1}}\sup\limits_{\bv_h\in\IR^{N_h-1}} \varrho(\bA_h,\bu_h,\bv_h),
\end{align*}
where $\varrho(\bA_h,\bu_h,\bv_h):= \frac{\bv_h^\top \bA_h \bv_h}{\norm{2}{\bu_h}\norm{2}{\bv_h}}$ denotes the Rayleigh quotient.

\begin{proposition}\label{prop:conditioning}
	Let \Cref{as:param,as:dParam} hold and let $V_h$ be spanned by piecewise linear FEs. Then, there exists a constant $c>$ independent of the mesh-size $h$ such 
	\begin{align*}
		\kappa(\bA_h)\leq c\, \frac{h^{-s}}{\alpha_{s,d}}.
	\end{align*}
\end{proposition}
\begin{proof}
	Let $u_h$, $v_h\in V_h$ denote arbitrary elements of the discrete space and denote the vectors associated with their FE representation as ${\bu}_h$ and ${\bv}_h\in\IR^{N_h-1}$, so that $a(u_h,v_h)=\bv_h^\top \bA_h \bu_h$. Then, $\sqrt{h} \| \bu_h\|_2 \sim \norm{L_2(\Omega)}{u_h}$ independent of $h$ for $u_h\in V_h$, where $\|\cdot\|_2$ denotes the standard Euclidean vector norm. Then, \Cref{thm:uniqueNonConst} yields 
\begin{align*}
	\sigma_{h,\text{min}}
		&= \inf\limits_{\bu_h\in\IR^n}\sup\limits_{\bv_h\in\IR^n} \frac{\bv_h^\top \bA_h\bu_h}{\norm{2}{\bu_h}\norm{2}{\bv_h}}
		\geq\inf\limits_{\bu_h\in\IR^n} \frac{\bu_h^\top \bA_h\bu_h}{\norm{2}{\bu_h}^2}\\
		&\gtrsim h \inf\limits_{u_h\in V_h} \frac{a(u_h,u_h)}{\norm{L_2(\Omega)}{u_h}^2}
		\gtrsim \alpha_{s,d}\, h \inf\limits_{u_h\in V_h} \frac{\norm{\hst{\sh}{\Omega}}{u_h}^2}{\norm{L_2(\Omega)}{u_h}^2}
		\gtrsim  \alpha_{s,d}\, h,
	\end{align*}
where we used e.g.\ \cite[La.\ 10.5]{steinbach2007numerical} and the involved constant is independent on the mesh-size. Next, proceeding analogously as before and employing \Cref{lem:nonConstCont}, we find that
	\begin{align*}
	\sigma_{h, \text{max}}
		&= \sup\limits_{\bu_h\in\IR^n}\sup\limits_{\bv_h\in\IR^n} \frac{\bv_h^\top \bA_h\bu_h}{\norm{2}{\bu_h}\norm{2}{\bv_h}}
		\lesssim h \, \sup\limits_{u_h\in V_h}\sup\limits_{v_h\in V_h} \frac{a(u_h, v_h)}{\norm{L_2(\Omega)}{u_h}\norm{L_2(\Omega)}{v_h}}\\
		&\lesssim h\, C_{d,r}\sup\limits_{u_h\in V_h}\sup\limits_{v_h\in V_h} \frac{\norm{\hst{\sh}{\Omega}}{u_h}\norm{\hst{\sh}{\Omega}}{v_h}}{\norm{L_2(\Omega)}{u_h}\norm{L_2(\Omega)}{v_h}}\\
		&\lesssim h\, C_{d,r}\, h^{-s}\sup\limits_{u_h\in V_h}\sup\limits_{v_h\in V_h} \frac{\norm{\lp{2}{\Omega}}{u_h}\norm{\lp{2}{\Omega}}{v_h}}{\norm{L_2(\Omega)}{u_h}\norm{L_2(\Omega)}{v_h}} 
		\lesssim  C_{d,r}\, h^{1-s},
	\end{align*}
    where we have used a Bernstein-type inequality, see e.g.\  \cite[La.\ 10.5, Thm.\ 10.7]{steinbach2007numerical}. The claim follows from the combination of the two previous results.
\end{proof}

%-------------------------------------------------------------------------------------------------------
\section{The Reduced Basis Method (RBM)}
\label{sec:rb}
%-------------------------------------------------------------------------------------------------------
In this section, we show how to apply the RBM for the above example of a \emph{parameterized} fractional PDE (PFPDE). We built upon standard references on the RBM such as \cite{Haasdonk:RB,Rozza:RB,Quarteroni:RB,Urban:RB}, refer the reader to these and omit further references as most of what is said is rather standard.

\subsection{A Parametrized Fractional PDE (PFPDE)} %Sort of an introduction
\label{ssec:pfpde}
%-------------------------------------------------------------------------------------------------------
The model reduction techniques can in general be used in order to find the numerical solutions of PDEs depending on parameters (PPDEs). These parameters are collected in a vector $\mu \in \IP$, with compact set $\IP \subset \mathbb{R}^P$, $P\in\IN$. Here, we consider a parametric version of \eqref{eq:PDE} for $\mu\in\IP$ defined via the parameter-dependent differential operator 
\begin{align*}
	\cA(\mu)u := -\flder{\sh}{\left(d(\mu)\, \flder{\sh}{u}\right)} + r(\mu)\, u,
\end{align*}
and right-hand side $f(\mu)$, where 
\begin{align*}
	c(\mu)= \sum_{q=1}^{Q^c} \vartheta_q^c(\mu)\, c_q,
	\qquad
	c\in\{ d,r,f\},
\end{align*}
$\vartheta_q^d,\ \vartheta_q^r,\ \vartheta_q^f:\IP\to\IR$, $Q^d,\ Q^r,\ Q^f\in\IN$ and $d_q,\ r_q\in\lp{\infty}{\Omega}$ as well as $f_q\in \lp{2}{\Omega}$ are given. The above representation of $d(\mu)$, $r(\mu)$ and $f(\mu)$ is known as \emph{affine decomposition}.  If such a representation is not given, one may construct an approximation by means of the \emph{Empirical Interpolation Method (EIM)}, \cite{EIM}. Then, we can write the (classical form of the) parametrized problem as
\begin{align}
	\cA(\mu)\, u_\mu(x) =f(x;\mu),\,\, x\in\Omega=(0,1), 
	\qquad
	u_\mu(0)=0, \, u_\mu(1)=0. 
	\label{eq:pPDE}
\end{align}
We shall use the above derived variational formulation of \eqref{eq:pPDE}: given $\mu \in \IP$, we seek $u_\mu \in V$ such that
\begin{equation}\label{eq:VarForm}
	a(u_\mu, v; \mu)= f(v;{\mu})  \qquad  \text{for all} \quad v \in V,
\end{equation}
where the parametric bilinear form $a: V\times V\times \IP\to\IR$ is defined as 
\begin{align*}
	a(u, v; \mu)
	&:= -\pr{\lp{2}{\Omega}}{d(\mu)\, \flder{\sh}{u}}{\frder{\sh}{v}} + \pr{\lp{2}{\Omega}}{r(\mu)\, u}{v}
\end{align*}
for $u,v \in V$ and the parametric linear form $f:H\times\IP \to V$, $H:=\lp{2}{\Omega}$, is given by $f(v;\mu)= \pr{H}{f(\mu;\cdot)}{v}$ for $v\in V\subset H$. 

Recall, that the well-posedness of the variational problem relies on \Cref{as:dParam}, see \Cref{thm:uniqueNonConst}.  By adding the parameter-dependency to the diffusion coefficient $d$ in  \Cref{as:dParam}, we obtain the parameter-dependent number $\gamma_{s,d}(\mu):= \average{d(\mu)}\,\abs{\cos\left(s\tfrac{\pi}{2}\right)} - \range{d(\mu)}$ and the condition \eqref{eq:Ass-d} then reads 
\begin{align}\label{eq:Ass-d-mu}
	\gamma_{s,d}(\mu) \tfrac{\Gamma(s/2+1)^2}{4} + \underline{r}(\mu) \ge 0. 
\end{align}
Correspondingly, the obtain a parameter-dependent coercivity constant, namely 
\begin{align}\label{eq:coer-mu}
	\alpha_{s,d}(\mu) = \gamma_{s,d}(\mu) \tfrac{\Gamma(s/2+1)^4}{8},
\end{align}
or the $r$-dependent variant
\begin{align}\label{eq:tildecoer-mu}
	\tilde\alpha_{s,d,r}(\mu) = \gamma_{s,d}(\mu) \tfrac{\Gamma(s/2+1)^2}{2} + \tfrac12\, \min\{ \underline{r}(\mu),0\}.
\end{align}

\subsection{A Detailed Discretization} %Sort of an introduction
%-------------------------------------------------------------------------------------------------------
The next standard ingredient for the RBM is a sufficiently detailed discretization in the sense that the numerical approximation is \enquote{indistinguishable} from the exact solution -- upon discretization errors, of course. This is why such an approximation is also called the \enquote{truth} solution, which is realized by a finite-dimensional subset $V^N\subset V$ of dimension $N\gg 1$, which is typically large. Then, the truth solution $u^N_\mu\in V^N$ is determined by solving the Galerkin problem
\begin{equation}\label{eq:truth}
	a(u^N_\mu, v^N; \mu)= f(v^N;{\mu})  \quad  \text{for all } \, v^N \in V^N,
\end{equation}
which is well-posed as long as the original problem \eqref{eq:VarForm} is so. The following error / residual-relation is well-known
\begin{align}\label{eq:error-residual}
  \norm{V}{u_\mu - u^N_\mu}
	\le \frac1{\alpha_{s,d}(\mu)}
		\sup_{v\in V}\frac{f(v;\mu) - a(u^N_\mu,v;\mu)}{\norm{V}{v}}
	=: \frac{\norm{V'}{\varrho^N_\mu}}{\alpha_{s,d}(\mu)},
\end{align}
where the residual is defined as usual, i.e., 
\begin{align*}
	\varrho^N_\mu(v):= f(v;\mu) - a(u^N_\mu,v;\mu)
	\quad
	\text{for }\, v\in V.
\end{align*}

\subsection{Greedy Reduced Basis Selection}
%-------------------------------------------------------------------------------------------------------
Finally, one determines a reduced space by selecting certain sample values $\mu^{(i)}$, $i=1,...,n$, in an offline training phase and computing the truth \enquote{snapshots}, i.e.,
\begin{align*}
	\xi_i := u^N_{\mu^{(i)}},\quad i=1,...,n\ll N.
\end{align*}
Once, these snapshots are determined, the reduced basis approximation for some given $\mu\in\IP$ is the Galerkin projection $u_n(\mu)$ of $u^N_\mu$ onto the reduced space $V_n:=\spann\{ \xi_1,...,\xi_n\}$, i.e.,
\begin{equation}\label{eq:RM}
  a(u_n(\mu), v_n; \mu)= f(v_n;{\mu})  \qquad  \text{for all} \quad v_n \in V_n,
\end{equation}
and we obtain a similar error estimate as in \eqref{eq:error-residual}
\begin{align}\label{eq:RB-error-residual}
  \norm{V}{u^N_\mu - u_n(\mu)}
	\le \frac{\norm{(V^N)'}{\varrho_n(\mu)}}{\alpha_{s,d}(\mu)}
	=: \Delta_n(\mu),
\end{align}
with the RB residual defined by $\varrho_n(v^N;\mu):= f(v^N;\mu) - a(u_n(\mu),v^N;\mu)$ for $v^N\in V^N$. 

The error estimator $\Delta_n(\mu)$ is maximized in the training phase over a finite-dimensional training set $\IP_{\text{train}}\subset\IP$ to determine the sample values $\mu^{(i)}$ of the parameters. it turns out that $\Delta_n(\mu)$ is computable with complexity independent of the truth dimension $N$, which is termed \enquote{online efficient}. In order to realize this, first note that the residual admits an affine decomposition, see \cite{Haasdonk:RB} for details. Then, the dual norm of the residual, i.e., $\| \varrho_n(\mu)\|_{(V^N)'}$ can be determined via the offline computation of Riesz representations as follows: given some $\varrho\in (V^N)'$, then its Riesz representation $\hat\varrho\in V^N$ is computed by solving
\begin{align}\label{eq:Riesz}
	\pr{V}{\hat\varrho}{v^N} = \varrho(v^N)
	\quad\text{for all}\quad v^N\in V^N,
\end{align}
and then $\norm{(V^N)'}{\varrho}=\norm{V^N}{\hat{\varrho}}$. Hence, we need to compute 
\begin{align*}
	\pr{V}{w}{v} = \pr{\hst{\sh}{\Omega}}{w}{v}.
\end{align*}
Owing to \Cref{Cor:NormEq}, we choose the inner product given by
\begin{align*}
	\pr{\lp{2}{\Omega}}{\flder{\sh}{w}}{\flder{\sh}{v}},
\end{align*}
inducing $\seminorm{s/2}{\cdot}$, which results in a symmetric and positive definite matrix.

\begin{remark}
	If one is not (or not only) interested to control the \emph{state} $u_\mu$, but some functions of it, a primal-dual RBM can be used following the same route as in the standard RB references quoted above. This can be used for fractional-type problems as well by using the above variational formulation of a PFPDE.
\end{remark}

\subsection{Kolmogorov $n$-Width}
\label{ssec:Kolnwith}
%--------------------------------------------------------------------
It is well-known that the error of the best possible approximation that can be achieved by the RBM is determined by the Kolmogorov $n$-width defined as 
\begin{align*}
	d_n(\cP):= \inf_{\substack{V_n\subset V\\ \dim(V_n)=n}} \sup_{\mu\in\cP} \inf_{v_n\in V_n} \| u(\mu) - v_n\|_V.
\end{align*}
Since the problem is coercive, we can immediately apply \cite[Thm.\ 3.1]{m.ohlbergers.rave2016} and conclude that
\begin{align*}
	d_n(\cP) \le C\, \exp\{ -c\, n^{1/(Q^d+Q^r)}\}
\end{align*}
for constants $C,c>0$ and $Q^d,\ Q^r\in\IN$ as in \Cref{ssec:pfpde}. The good news is that we may expect exponential decay as $n$ grows. However, typically the sizes of the involved constants are not known. Moreover, we have to expect that the rate deteriorates as $s\to 1$ since the fractional problem then becomes more and more like a transport problem, where $d_n(\cP)$ is expected to decay slowly, \cite{AGU25,m.ohlbergers.rave2016}. Hence, we are interested in making the dependence of the decay of $d_n(\cP)$ w.r.t.\ the fractional order $s$ explicit.

\begin{theorem}\label{thm:Koln}
	Let $\cA_s(\mu)u:=-\flder{s}u + \mu\, u=f$ (in variational form) for $\mu\in\cP:=[0,\mu^+]$ and some $0<\mu^+<\infty$. Then, there exists a constant $c$ and a constant $C_s$ depending on $s$ with   $C_s\stackrel{s\to 1}{\longrightarrow}\infty$, such that 
\begin{equation*}
	d_n(\cP) 
	\le C_s\, \exp\Big\{ 
		-c \, \frac{\alpha_s n}{|\cP|}
		\Big\},
\end{equation*}
		where $\alpha_s\equiv\alpha_{1,s}$ is the coercivity constant of $-\flder{s}{}$ (i.e., for $d\equiv 1$).
\end{theorem}
\begin{proof}
We start by considering the complexification, i.e., $\cA^\IC_s(\hat\mu)\hat{u}:=-\hflder{s}{\hat{u}} + \hat\mu\, \hat{u} = \hat{f}$ for $\hat\mu = \nu +\iu \sigma$, $\hat{u} = u_1+\iu u_2$, $\hat{f} = f_1+\iu f_2$. This can rewritten as a system of the form $a_s(U,V)+ (C(\hat\mu)\, U,V)_{L_2(\Omega)^2}$, for $U=(u_1,u_2)^T$, $V=(v_1,v_2)^T$, where $a_s(U,V):= -(\flder{\sh}{u_1},\frder{\sh}{v_1})_{L_2(\Omega)}-(\flder{\sh}{u_2},\frder{\sh}{v_2})_{L_2(\Omega)}$ nad $C(\hat\mu):=\begin{pmatrix} \nu & -\sigma \\ \sigma & \nu\end{pmatrix}$. Then, $a_s(U,U) + \nu\,  (U,U)_{L_2(\Omega)^2}\ge \alpha_s \| U\|_{\hst{\sh}{\Omega}}^2  + \nu\, \| U\|_{L_2(\Omega)^2}^2$, which is positive for $\nu=\Re(\hat\mu)>-\alpha_s$. In that case, the operator $\cA^\IC_s(\hat\mu)$ is invertible.

For $V:=\hst{\sh}{\Omega}$ we define the solution map $\Phi_s:\cP\to V$ by $\Phi_s(\mu):= u_s(\mu)= \cA_s(\mu)^{-1}f$, where $u_s(\mu)$ denotes the unique solution of the variational formulation.  
Based upon the above considerations, we may even allow complex-valued parameters. Hence, we consider the complexifications  of these operators and define $\Psi_s: V^\IC\times \IC \to (V^\IC)'$ by $\Psi_s(\hat{v},\hat\mu) := \cA_s^\IC(\hat\mu)\, \hat{v} - f^\IC$, where $f^\IC\in (V^\IC)'$ denotes the complexification of $f\in V'$. Note, that we can rewrite the variational form as $\Psi_s(\Phi_s(\mu),\mu)=0$.  Recall that, the operator $\partial_v \Psi_s(v,\hat\mu) = \cA_s^\IC(\hat\mu)$ is invertible for all $\hat\mu\in\hat\cP:=\{ \hat\mu\in\IC:\, \Re(\hat\mu)>-\alpha_s\}$, which contains $\cP$. 
As in \cite{m.ohlbergers.rave2016} (and previously in \cite{a.cohenr.devore2015}) we use the  complex Banach space version of the implicit function theorem to deduce that the mapping $\hat\Phi_s:\hat\cP\to V^\IC$ defined as $\hat\Phi_s(\hat\mu):=\Phi_s(\Re(\hat\mu))$, can holomorphically be extended to an open neighborhood $\hat\cP\subset\hat\cO\subset\IC$ (which is the reason for the notation $\hat\Phi_s$ by which we shall denote the extension to $\hat\cO$, which still can be chosen appropriately).  
Now, we  cover $\hat\cP$ by circles. In fact, we choose centers $c_m:= m\, \alpha_s$, $m=1,...,M_s(\mu^+):= \Big\lceil \tfrac{\mu^+}{\alpha_s}\Big\rceil$ and fix the radius $r:=(1+\varepsilon)\alpha_s$ so that 
\begin{align*}
	\hat\cP \subset \bigcup_{m=1}^{M_s(\mu^+)} D(c_m,r)
	\qquad\text{and}\qquad
	\bigcup_{m=1}^{M_s(\mu^+)} D(c_m,2r) \subset\hat\cO,
\end{align*}
where $D(c,r):=\{ z\in\IC^2:\, |z-c|<r \}$ is the ball of radius $r$ around $c$. 
Next, holomorphy implies analyticity, so that for all $1\le m\le M_s(\mu^+)$ and all indices $\nu\in\IN_0$ there exist functions $v_{m,\nu}\in V^\IC$ such that
\begin{align*}
	\hat\Phi(z) = \sum_{\nu\in\IN_0} (z-c_m)^\nu\, v_{m,\nu}
\end{align*}
converges absolutely for all $z\in D(c_m,2r)$. Recall, that $v_{m,\nu} = \frac{\partial^\nu}{\partial z^\nu} \hat\Phi_s(c_m)$ and $\hat\Phi_s(c_m)=\Phi_s(c_m)=u_s(c_m)$ as well as $u_s(\mu)=\cA_s(\mu)^{-1}f$ and $\| \cA_s(\mu)^{-1}\| \le \alpha_s^{-1}$ by coercivity. Hence, we deduce that $v_{m,\nu}\in V$ and $\| v_{m,\nu}\| \lesssim \alpha_s^{-\nu}$. Thus, the following quantity is independent of $M_s(\mu^+)$
\begin{align*}
	C_s&:= \max_{1\le m\le M_s(\mu^+)} \sup_{z\in D(c_m,r)}
		\bigg\| \sum_{\nu\in\IN_0}  2^\nu (z-c_m)^\nu\, v_{m,\nu} \bigg\|\\
	&\le \max_{1\le m\le M_s(\mu^+)} \sum_{\nu\in\IN_0} 2^\nu 2^{-\nu} \| v_{m,\nu}\|
		<\infty.
\end{align*}
The term $2^{-\nu}$ can be put here since the supremum if taken over $z\in D(c_m,r)$, i.e., the half radius compared to the area of holomorphy which has a radius of $2r$. Moreover as $\alpha_s\to 0$ as $s\to 1$, we also see that $C_s$ grows as $s\to 1$. 
Now, set
\begin{align*}
	V_n := \text{span}\left\{ v_{m,\nu}:\, m=1,...,M_s(\mu^+), 
	\nu = 0,...,K_{n}
	:=\Big\lfloor {\textstyle\frac{n}{M_s(\mu^+)}}\Big\rfloor\right\},
\end{align*}
which is a subspace of $V$ of dimension at most $n\in\IN$. Moreover, for each $\mu\in\cP$, we find a $m\in\{ 1,...,M_s(\mu^+)\}$ such that $\mu\in D(c_m,r)$ and define the reduced approximation by the truncated series 
\begin{align*}
	\Phi_{s,n}(\mu) 
	:= \sum_{|\nu|\le K_n} (\mu-c_m)^\nu\, v_{m,\nu}\in V_n.
\end{align*}
Finally, we estimate the error
\begin{align*}
	\| \Phi_s(\mu)-\Phi_{s,n}(\mu)\|
	&= \bigg\| \sum_{\nu\ge K_{n}+1} (\mu-c_m)^\nu\, v_{m,\nu} \bigg\|
	= \bigg\| \sum_{\nu\ge K_{n}+1}  2^{-\nu} 2^{\nu} (\mu-c_m)^\nu\, v_{m,\nu} \bigg\|\\
	&\le C_s 2^{-(K_{n}+1)}
	\le C_s\, \exp\bigg\{ -\frac{n\, \ln(2)}{M_s(\mu^+)} \bigg\},
\end{align*}
which proves the claim.
\end{proof}

\Cref{thm:Koln} shows that the decay deteriorates as $s\to 1$. In fact, since the constants $C_s$ and $c_\Omega$ are independent of $\cP$, we clearly see the influence of the fractional order $s$ in the sense that it deteriorates the exponential decay rate, i.e., the convergence becomes slower and slower as $s\to 1$, i.e., $\alpha_s\to 0$ and $C_s$ grows. 

With some technical effort and corresponding assumptions, the statement of \Cref{thm:Koln} can also be extended to operators of the form 
\begin{align*}
	\cA_s(\mu) = \sum_{q=1}^{Q^d} \vartheta_q^d(\mu)\, D_{s,q} + \sum_{q=1}^{Q^r} \vartheta_q^r(\mu)\, r_q.
\end{align*}

%-------------------------------------------------------------------------------------------------------
\section{Numerical Experiments}
\label{sec:numexp}  
%-------------------------------------------------------------------------------------------------------
In this section, we present numerical experiments complementing our theoretical results presented in \Cref{sec:fem,sec:rb} by quantitative statements. All simulations were carried out using a C++ code developed by the authors \cite{GitAylwin}. The open source libraries Eigen \cite{eigenweb} and BOOST \cite{boost_math_quadrature,schaling2014boost} were used for the treatment of the linear systems and numerical integration, respectively. 

Based upon \Cref{Cor:NormEq}, we computed $\seminorm{s}{\varphi}=\norm{\lp{2}{\IR}}{\flder{s}{\varphi}}$ as a quantity being equivalent to $\norm{\hst{s}{\Omega}}{\varphi}$.

\subsection{FEM - Solutions and Convergence}
\label{ssec:femNumExp}
%-------------------------------------------------------------------------------------------------------
We begin with numerical experiments showcasing the convergence result in \Cref{thm:FEMerror}.

\subsubsection{Constant Coefficients}\label{sssec:NumExpConst}
%-------------------------------------------------------------------------------------
We begin by considering, as in \cite[\S 7.1]{jin}, \Cref{prob:nonconstCoeffDis} with $d\equiv 1$ and $r\equiv 0$, and test two different sources for which the analytical solution is known, namely
\begin{compactenum}[(Ex1)]
	\item $f(x) := 1$, for which the solution is $u(x):= \frac{1}{\Gamma(s+1)}(x^{s-1}-x^{s})$,
	\item $f(x) := x(1-x)$, yielding $u(x):= \frac{1}{\Gamma(s+2)}(x^{s-1}-x^{s+1})- \frac{1}{\Gamma(s+3)}(x^{s-1}-x^{s+2})$.
\end{compactenum}
We also consider, for each case, three different values of $s$, namely $s_1=1.8$, $s_2=1.5$ and $s_3=1.2$, for which the expected convergence rates in $\hst{\sh}{\Omega}$ read $\beta_i=\frac{s_i}{2}-\frac12$ given by \Cref{thm:FEMerror} are $\beta_1=0.4$, $\beta_2=0.25$ and $\beta_3=0.1$ (with respect to the mesh-size). The rates in $\lp{2}{\Omega}$ should be the double, i.e., $\nu_i=2\beta_i=s_i-1$, i.e., $\nu_1=0.3$, $\nu_2=0.5$ and $\nu_3=0.2$. 

\Cref{fig:femConv1,fig:femConv2} display the numerical convergence rates for the two cases of $f$. As expected, the  experimental convergence rates in $\hst{\sh}{\Omega}$ asymptotically reach the predicted theoretical rates. However, the experimental convergence rate in the $\lp{2}{\Omega}$ norm is higher than that displayed in \Cref{rmk:convL2} (also see \cite[Thm.\ 5.3]{jin}). In fact, instead of $\nu_i=2\beta_i=s_i-1$, we observe a rate of  $s_i-\frac{1}{2}$. This effect was also observed and commented in \cite[\S\S 7.1.1, 7.1.2]{jin} but (to the best of our knowledge) there is no analysis explaining this behavior. We also observe a pre-asymptotic range of faster convergence that appears to depend on the fractional order $s$ (with a longer pre-asymptotic range for larger values of $s$).
%-------------------------------------------------------
\begin{figure}[htbp]
  \centering
  \includegraphics[width=0.9\textwidth]{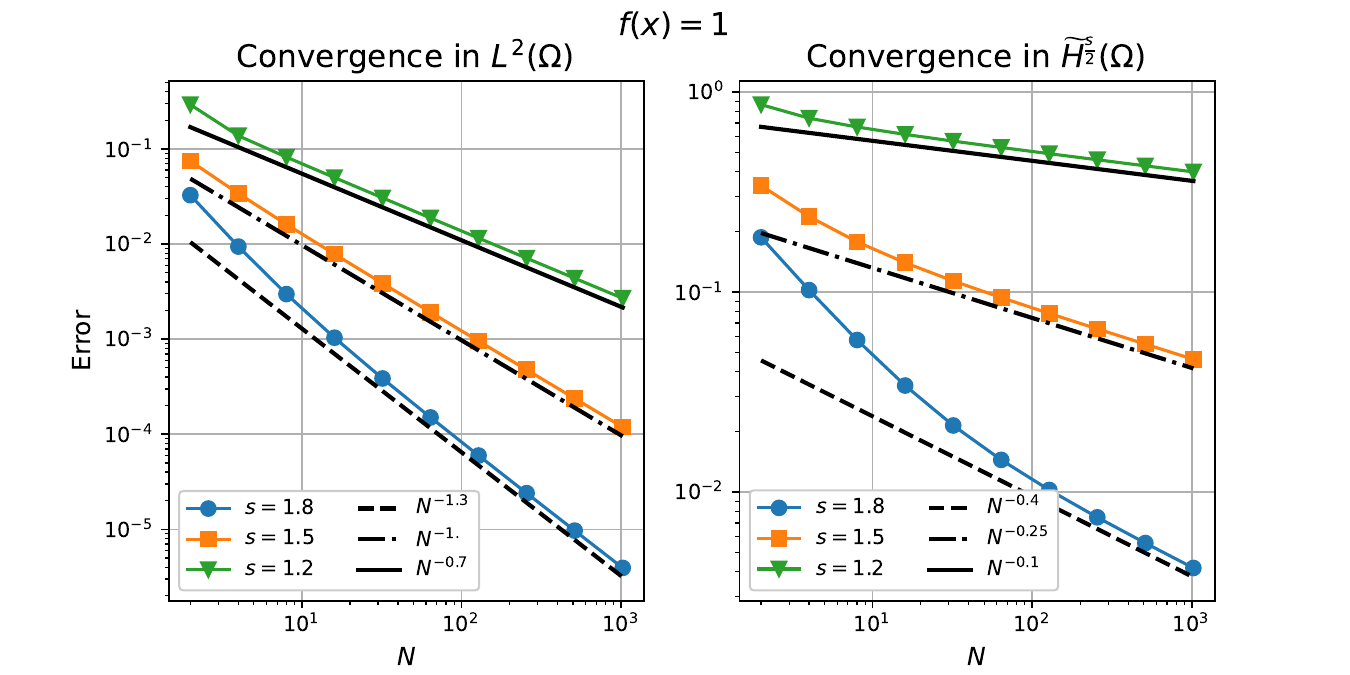}
  \caption{Convergence in the $\lp{2}{\Omega}$ and $\hst{\sh}{\Omega}$ norms to the analytical solution for the case $f(x) = 1$, (Ex1).}
  \label{fig:femConv1}
\end{figure}
%-------------------------------------------------------
\begin{figure}[htbp]
  \centering
  \includegraphics[width=0.9\textwidth]{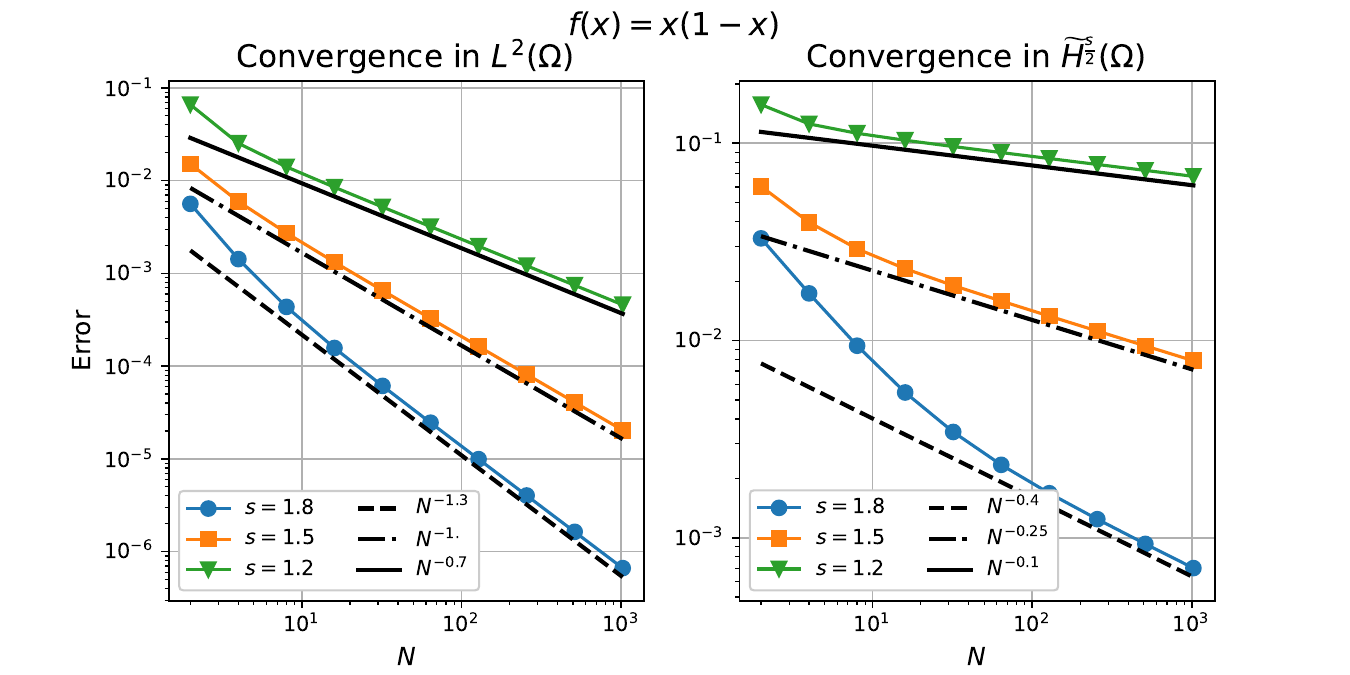}
  \caption{Convergence in the $\lp{2}{\Omega}$ and $\hst{\sh}{\Omega}$
    norms to the analytical solution for the case $f(x) = x(1-x)$, (Ex2).}
  \label{fig:femConv2}
\end{figure}
%-------------------------------------------------------

\subsubsection{Non-Constant Coefficients}
\label{sssec:NumExpNonConst}
%-------------------------------------------------------------------------------------
Next, we study the convergence of solutions to \Cref{prob:nonconstCoeffDis} with non-constant diffusion and reaction coefficients $d$ and $r$, respectively. 
We consider two different cases, namely
\begin{compactenum}[(Ex1)]
	\setcounter{enumi}{2}
	\item $d(x) := 4+\sin(2\pi x)$ and $r(x):= \cos(2\pi x)$,
	\item $d(x) := \begin{cases} 5, & \text{for }x < 0.5,\\ 3, &\text{for }x \geq 0.5,\end{cases}$ 
		\hspace*{3mm}and\hspace*{3mm} 
		$r(x):= \begin{cases} -2, &\text{for }x < 0.5,\\ 8, &\text{for }x \geq 0.5.\end{cases}$
\end{compactenum}
Apparently, (Ex3) is made of smooth coefficients, whereas (Ex4) are piecewise defined functions. With reference to \Cref{as:dParam}, we get $\average{d}=4$ and $\range{d}=1$ in both cases, so that  $\gamma_{s,d}:= 4\,\abs{\cos\left(s\tfrac{\pi}{2}\right)} - 1$. This also means that the coercivity constants are identical in both cases. However, the reaction coefficients vary, in the first example we have $\underline{r} =-1$ and $\underline{r} =-2$ in the second one. This yields the fact the first example is only coercive for $s=1.8$ and $s=1.5$, whereas the second one only for $s=1.8$, see \Cref{Tab:Ex2Constants}.
%-------------------------------
\begin{table}[!htb]
	\begin{center}
	\begin{tabular}{r|r|r|r|r}
	\multicolumn{5}{c}{(Ex3)} \\ \hline
		$s$ & $\gamma_{s,d}$ 
		& $c_{s,d,r}$ 
		& $\alpha_{s,d}$ 
		& $\tilde\alpha_{s,d,r}$ \\ \hline
		1.8 & 2.80 & 1.59 & 0.2999 & 0.7969 \\ \hline
		1.5 & 1.83 & 0.54 & 0.1631 & 0.2722\\ \hline
		1.2 & 0.24 & -0.81 & 0.0188 & -0.4058 \\ \hline
	\end{tabular}
	\hfil
	\begin{tabular}{r|r|r|r}
		\multicolumn{4}{c}{(Ex4)} \\ \hline
		$s$  & $c_{s,d,r}$ & $\alpha_{s,d}$ & $\tilde\alpha_{s,d,r}$ \\ \hline
		1.8 & 0.59 & 0.2999 & 0.2969 \\ \hline
		1.5 &  -0.46 & 0.1631 & -0.2278 \\ \hline
		1.2 &  -1.81 & 0.0188 & -0.9058 \\ \hline
	\end{tabular}
\caption{\label{Tab:Ex2Constants}Involved constants for the two examples of non-constant coefficients. Left (Ex3), right (Ex4). The values for $\gamma_{s,d}$ coincide for (Ex3) and (Ex4).}
\end{center}
\end{table}

We take the right-hand side $f(x) \equiv 1$ for the experiments in this subsection and compute the error against a numerically computed solution on a fine mesh. \Cref{fig:femSol3} shows the numerical solutions. The growth of the strength of the singularity with decreasing $s$ is apparent in both figures, while the solutions in the right graph in \Cref{fig:femSol3} display a visible dent at the point $x=0.5$, where the coefficients are discontinuous.
%-------------------------------------------------
\begin{figure}[!htbp]
  \centering
  \includegraphics[width=0.4\textwidth]{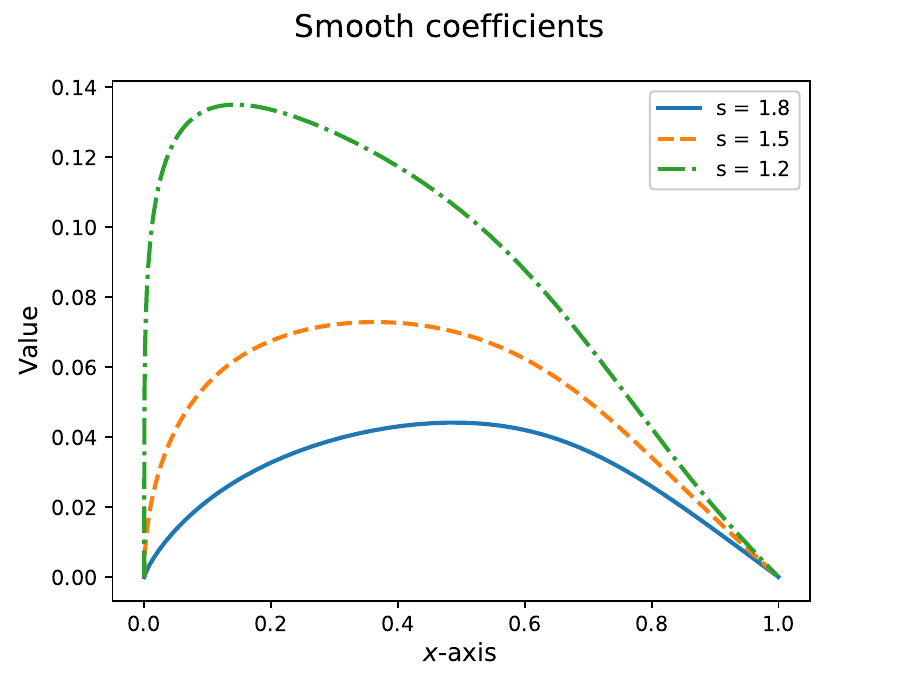}
  \includegraphics[width=0.4\textwidth]{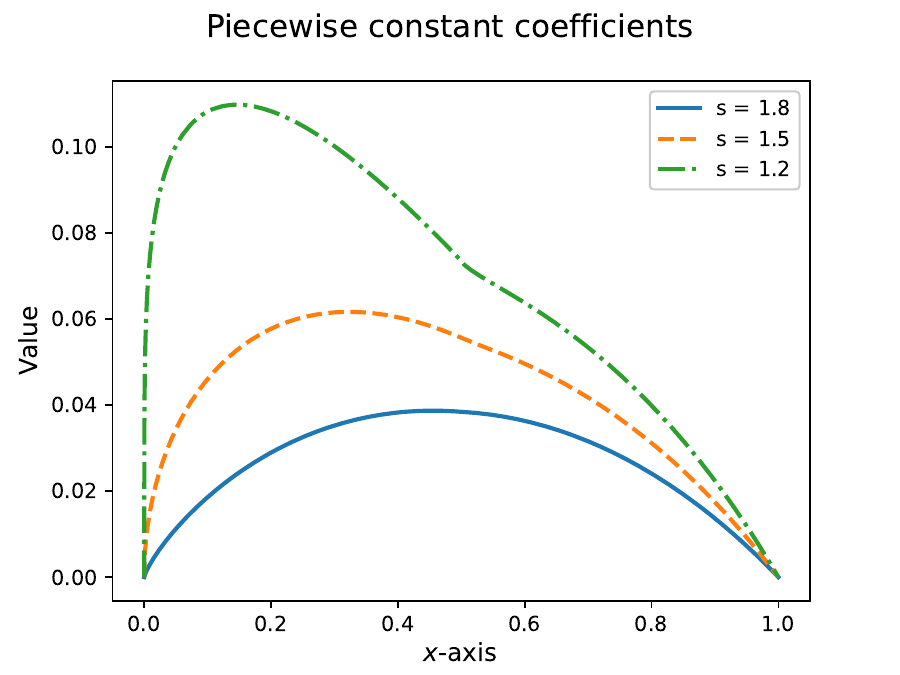}
  \caption{Numerical solutions for the case with smooth coefficients (Ex3, left) and piecewise constant coefficients (Ex4, right) for different values of $s$.}
  \label{fig:femSol3}
\end{figure}

\Cref{fig:femConv3,fig:femConv4} display the convergence of the discrete solutions obtained through our implementation to the respective numerically computed solutions. The figures show a behavior similar to the constant coefficient case in \Cref{sssec:NumExpConst}, with a pre-asymptotic range of faster convergence. The smoothness of the coefficients seems to play no role in the observed convergence rates, which are dominated by the singularity at the $x=0$ endpoint. Moreover, we obtain also convergence in those cases where the bilinear form is no longer coercive, see \Cref{Remark:infsup}.
%-----------------------------------------------
\begin{figure}[!htb]
  \centering
  \includegraphics[width=0.9\textwidth]{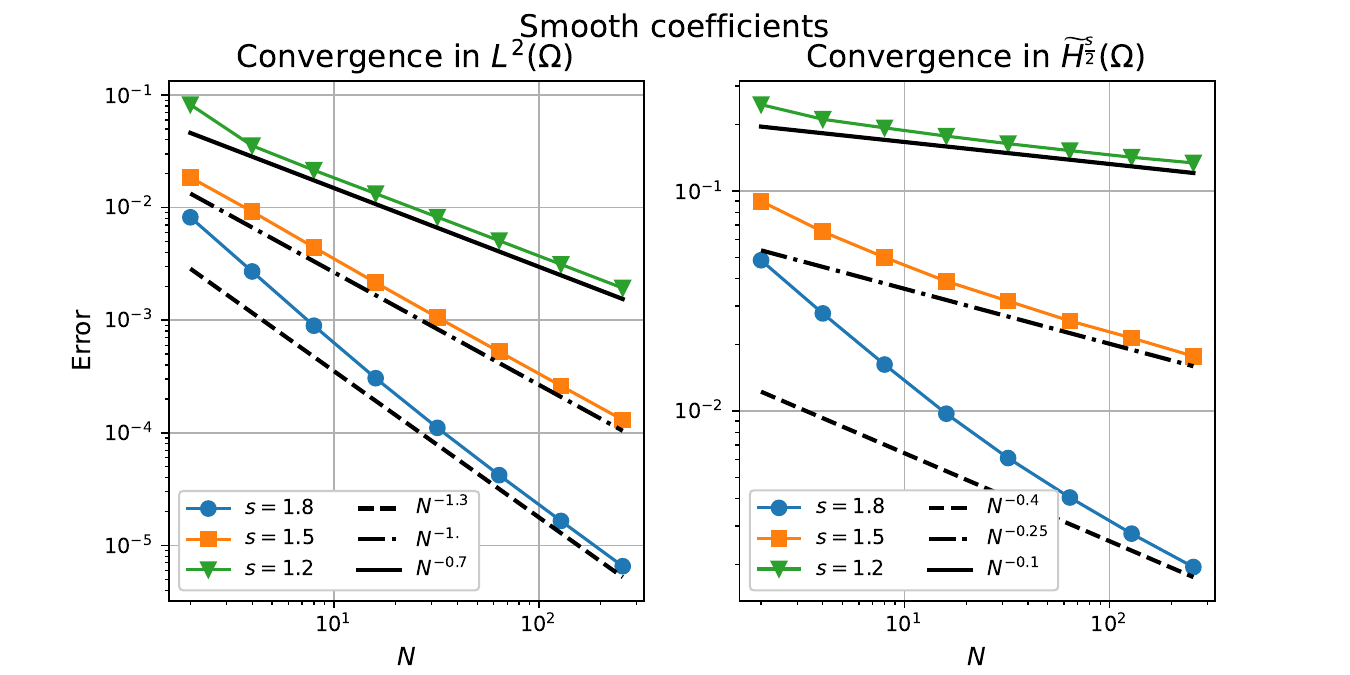}
  \caption{Convergence in the $\lp{2}{\Omega}$ and $\hst{\sh}{\Omega}$
    norms to the analytical solution for the case with smooth coefficients, (Ex3).}
  \label{fig:femConv3}
\end{figure}
%-----------------------------------------------
\begin{figure}[htbp]
  \centering
  \includegraphics[width=0.9\textwidth]{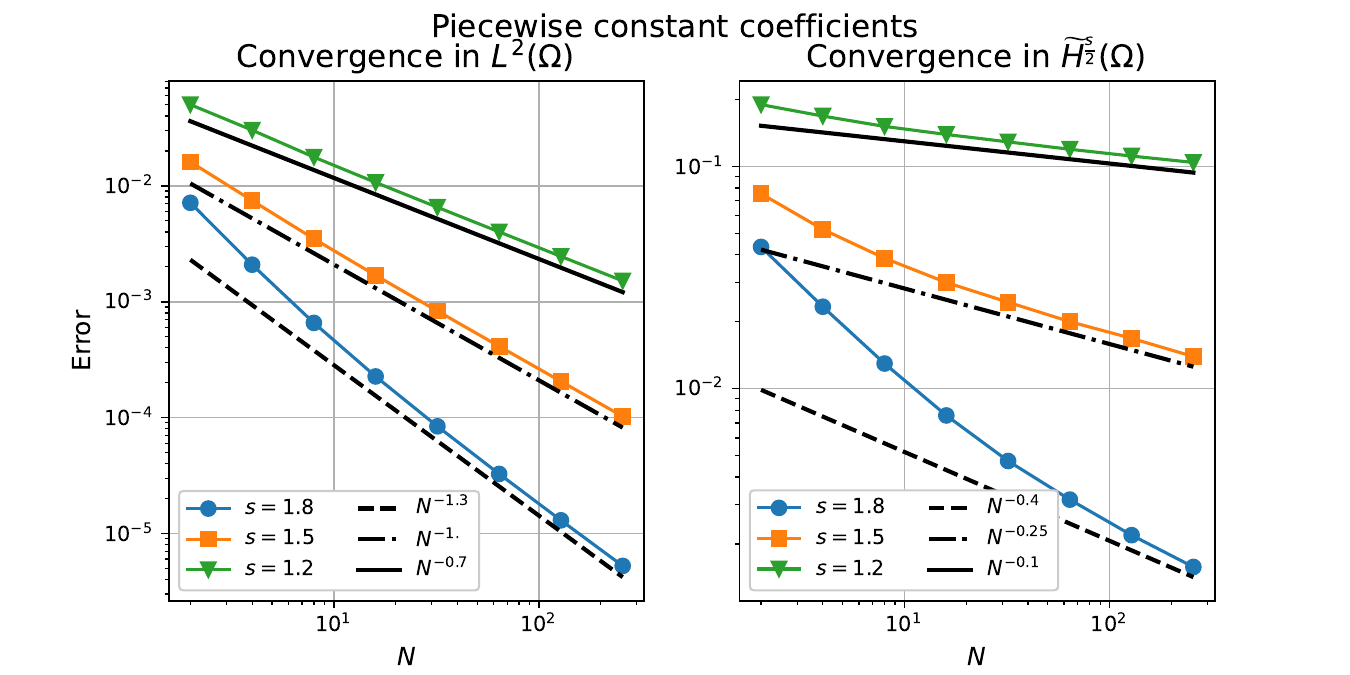}
  \caption{Convergence in the $\lp{2}{\Omega}$ and $\hst{\sh}{\Omega}$ norms to the analytical solution for the case with piecewise constant coefficients, (Ex4).}
  \label{fig:femConv4}
\end{figure}
%-----------------------------------------------
\subsection{Conditioning of the Stiffness Matrix}
% -----------------------------------------------
We continue by presenting numerical experiments concerning \Cref{prop:conditioning}. To this end, we will consider the stiffness matrices for three different problems, namely,
\begin{compactenum}
	\item $A_h^{(1)}$ from (Ex1) or (Ex2) in \Cref{sssec:NumExpConst};
	\item $A_h^{(2)}$ corresponding to (Ex3) in \Cref{sssec:NumExpNonConst};
	\item $A_h^{(3)}$ corresponding to (Ex4) in \Cref{sssec:NumExpNonConst}.
 \end{compactenum}
\Cref{fig:singVals} displays the behavior of the singular values as the number of elements of the mesh increases (i.e., $h\to 0$). The maximum singular values display the behavior $N^{s-1}$, while the minimum singular values decrease as $N^{-1}$, as the upper bound found in \Cref{prop:conditioning}. The influence of the coercivity constant is visible by the vertical offset of the curves corresponding to the same matrix.
%-----------------------------------------------
\begin{figure}[htbp]
  \centering
  \includegraphics[width=0.9\textwidth]{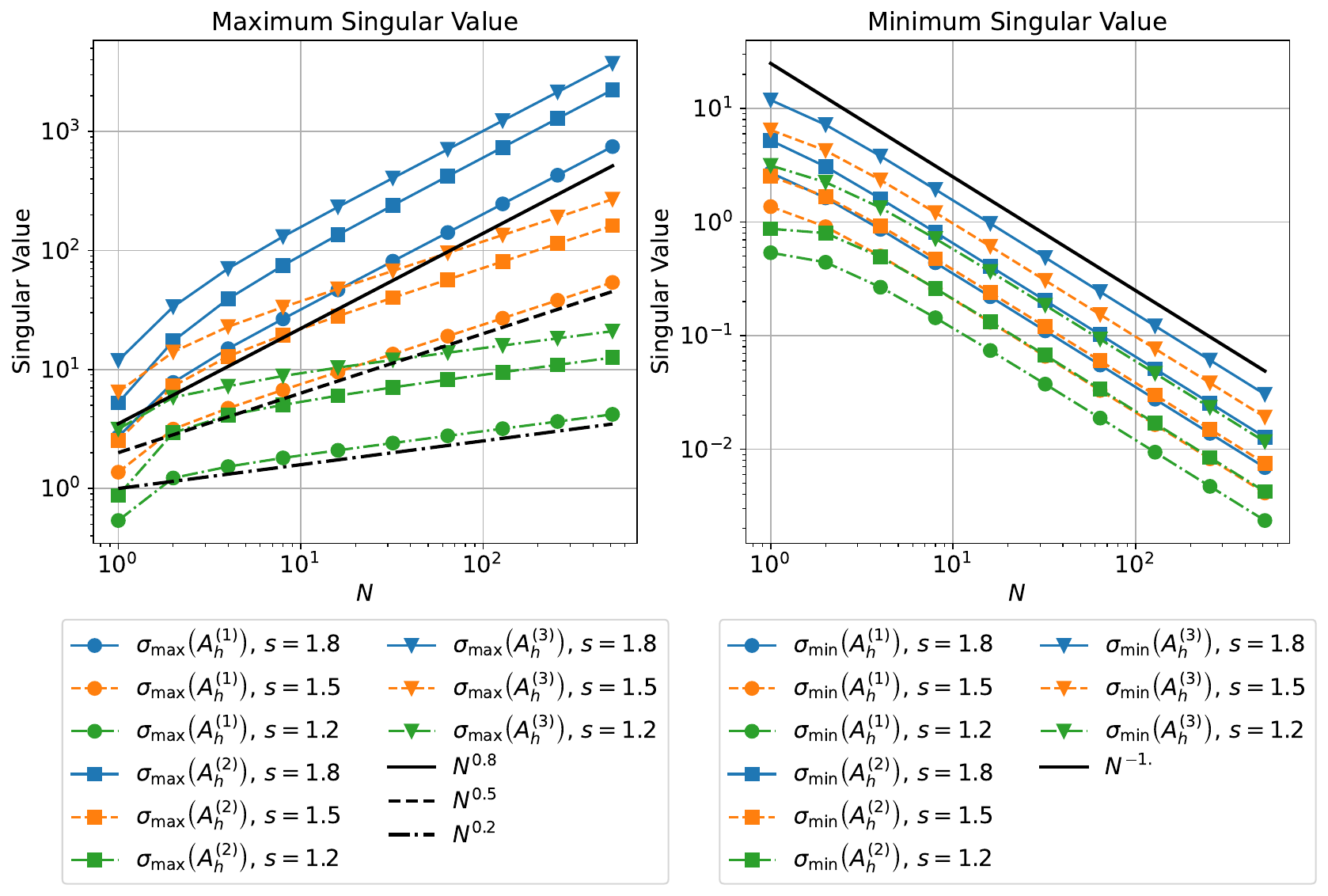}
  \caption{Maximum and minimum singular values for the different stiffness matrices and fractional orders under consideration.}
  \label{fig:singVals}
\end{figure}
%-----------------------------------------------

%-----------------------------------------------
\subsection{Reduced Basis Method}
%-----------------------------------------------
Finally, we present results of our numerical experiments concerning the RBM. Our specific interest lies in the quantitative investigation of the reduction keeping in mind that the original fractional problem is nonlocal. Hence, we are interested in the decay of the greedy scheme w.r.t.\ the size $n$ of the reduced system as well as in the speedup of the reduced system as compared to the full \enquote{truth} problem. For both issues, we aim at studying the dependence on the order $s$ of the operator.

We restrict ourselves to a parameter-independent right-hand side  $f(x):=x(1-x)$ and consider a parameter-dependent piecewise constant diffusion and convection coefficients given by
\begin{align}\label{eq:greedyCase1}
	d(x; \mu) := \begin{cases}
		1, 		& \text{if}\quad x\in[0, 0.25),\\
		\mu_1, 	& \text{if}\quad x\in[0.25, 0.5),\\
		\mu_2, 	& \text{if}\quad x\in[0.5, 0.75),\\
		\mu_3, 	& \text{if}\quad x\in[0.75, 1],
		\end{cases}
		\qquad\text{and}\quad r(x; \mu):= \mu_4 + \mu_5 x.
\end{align}
We choose $\mu\in \cP:=[0.7, 1.3]^3\times[0, 1]^2\subset\IR^5$ and use the training set as the tensor product of 10 Gauss-Legendre quadrature points in each dimension, yielding $|\IP_{\text{train}}|=10^5$ training points. Again, we choose $s=1.8$, $1.5$ and $1.2$. The \enquote{truth} approximatios are computed through the FEM on a uniform mesh with $2^9$ elements. Since we consider here only cases where $\underline{r}\geq 0$ and we employ $\seminorm{s/2}{\cdot}$ as an equivalent norm on $\hst{\sh}{\Omega}$, we use $\alpha_{s,d}$ as the associated coercivity constant (see the proof of \Cref{thm:uniqueNonConst}).

\subsubsection{Greedy Convergence}
%-------------------------------------------------------------------------------
We determine the decay of the error of the strong greedy method for two cases, namely
\begin{compactenum}
	\item constant diffusion ($d\equiv 1$), for which the bound in \Cref{thm:Koln} holds;
	\item the more general non-constant diffusion case \eqref{eq:greedyCase1}.
\end{compactenum}
In both cases, we expect rates of the form $d_n(\cP) \le C_s e^{-n\, \varrho_s}$. At least the constant $C_s$ cannot easily be computed and also $\varrho_s$ is not immediately accessible. Hence, we are interested in deriving experimental estimates for $\varrho_s$.

\textbf{Case 1, constant diffusion:}
For this case, \Cref{thm:Koln} yields $\varrho_s=c_\Omega \, \frac{\alpha_s}{|\cP|}$, where here $|\cP|=1$ and $\alpha_s = |\!\cos(s\frac\pi2)|$, i.e., $\alpha_{1.8}=0.95$, $\alpha_{1.5}=0.71$ and $\alpha_{0.12}=0.31$. \Cref{fig:greedy1} shows the convergence history of the greedy algorithm through the dual norm of the residual compared with the actual $\hst{\sh}{\Omega}$-error. We observe exponential rates of convergence of $\varrho^{\text{obs}}_{1.8} = 4.8$, $\varrho^{\text{obs}}_{1.5}=4.2$ and $\varrho^{\text{obs}}_{1.2}=3.3$, which are larger than the corresponding value of $\alpha_s$ -- for $s\in\{ 1.5, 1.8\}$ by a factor of about $5$, for $s=1.2$ by a factor of about 10. We also see that the dual nom of the residual is quite close to the $\hst{\sh}{\Omega}$-error. However, for smaller $s$, the gap between these two values increases. If we estimate the unknown constant $c_\Omega\approx 5$ from the observed convergence rate for $s=1.8$ and calculate the expected convergence rates at $s=1.5$ and $s=1.2$, we get $\varrho_{1.5}=3.6<4.2=\varrho^{\text{obs}}_{1.5}$ and $\varrho_{1.2}=1.6<3.3=\varrho^{\text{obs}}_{1.2}$, i.e., the observed rates of convergence are larger than their estimates from theory.  Since this effect cannot totally be explained by the unknown constant $C_s$ (which grows as $s\to 1+$, but whose value is unknown), we suppose a  a deficiency in the estimate of the coercivity constant in \Cref{thm:uniqueNonConst} for smaller values of $s$.
%-------------------------------------------------------------------------------
\begin{figure}[htbp]
  \centering
  \begin{subfigure}{0.32\textwidth}
    \includegraphics[width=\textwidth]{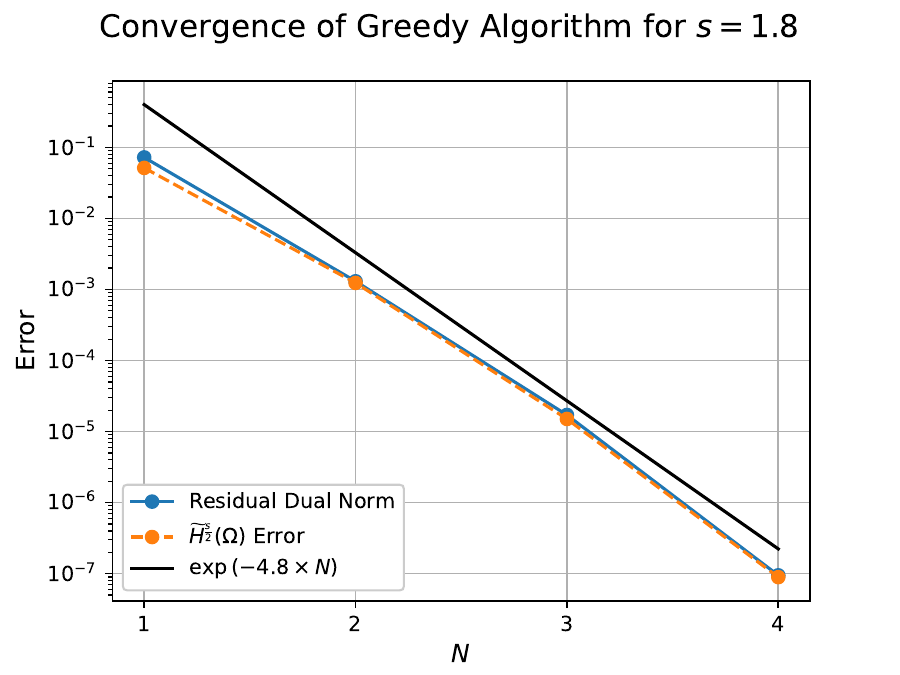}
    \caption{$s=1.8$: $e^{-4.8 n}$.}
    \label{subfig:greedy2_18}
  \end{subfigure}
  \hfill
  \begin{subfigure}{0.32\textwidth}
    \includegraphics[width=\textwidth]{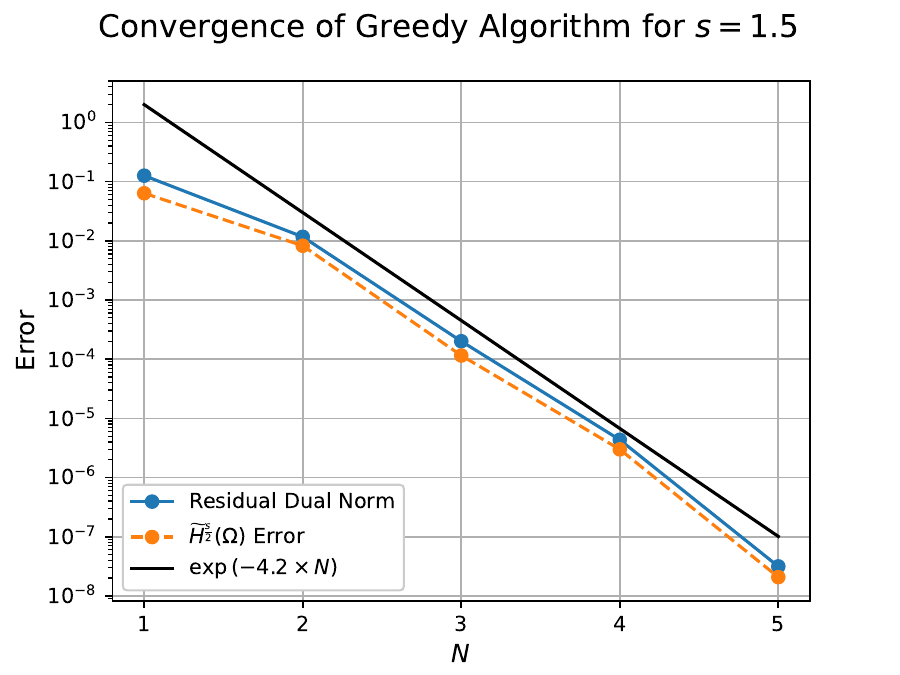}
    \caption{$s=1.5$: $e^{-4.2 n}$.}
    \label{subfig:greedy2_15}
  \end{subfigure}
  \hfill
  \begin{subfigure}{0.32\textwidth}
    \includegraphics[width=\textwidth]{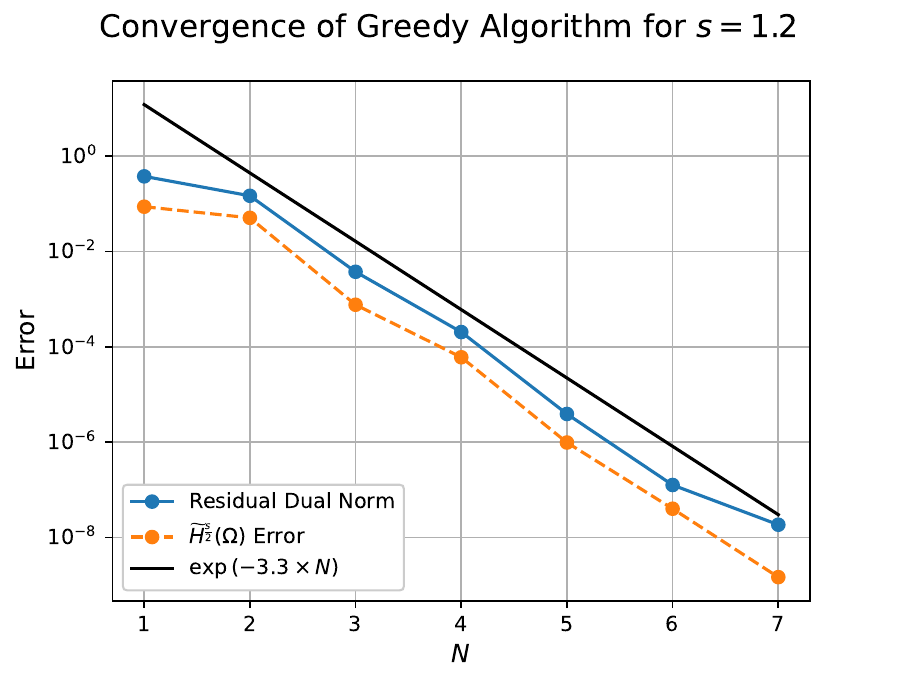}
    \caption{$s=1.2$: $e^{-3.3 n}$.}
    \label{subfig:greedy2_12}
  \end{subfigure}
  \caption{Convergence history of the Greedy Algorithm for the PFPDE with constant diffusion for $s\in\{1.8,1.5,1.2\}$.}
  \label{fig:greedy1}
\end{figure}
%-------------------------------------------------------------------------------

\textbf{Case 2, parameterized diffusion:} For the more general case of non-constant diffusion coefficient in \eqref{eq:greedyCase1}, we cannot apply the estimate in  \Cref{thm:Koln}.  \Cref{fig:greedy2} displays the strong greedy convergence for the different cases of $s$. We observe exponential rates $\varrho^{\text{obs}}_{1.8}=0.96$, $\varrho^{\text{obs}}_{1.5}=0.76$ and $\varrho^{\text{obs}}_{1.2}=0.65$, which (again) worsen as $s$ approaches $1$. We also display the dual norm of the residual and the actual $\hst{\sh}{\Omega}$ error. Again, the gap between the two displays a dependence on the order $s$, worsening with decreasing values of $s$. This effect is even more pronounced as in Case 1. Overall, the rates are worse than in the pure reaction-case and for $s\to 1+$, we even expect to loose exponential decay (which is perfectly inlined with the known effects for the RBM for transport problems, \cite{AGU25,m.ohlbergers.rave2016}).
%-------------------------------------------------------------------------------
\begin{figure}[htbp]
  \centering
  \begin{subfigure}{0.32\textwidth}
    \includegraphics[width=\textwidth]{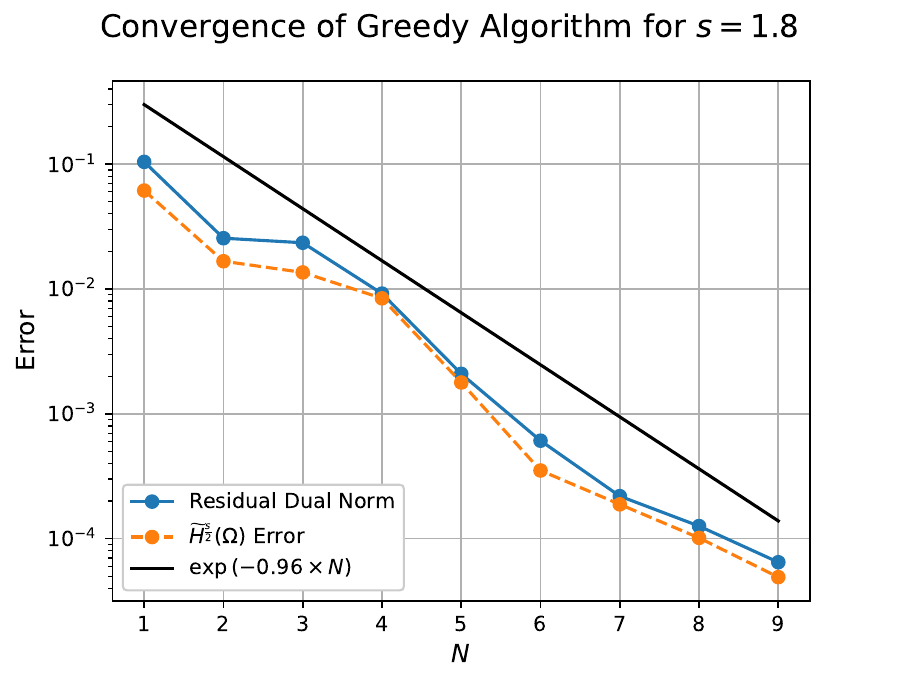}
    \caption{$s=1.8$: $e^{-0.96 n}$.}
    \label{subfig:greedy1_18}
  \end{subfigure}
  \hfill
  \begin{subfigure}{0.32\textwidth}
    \includegraphics[width=\textwidth]{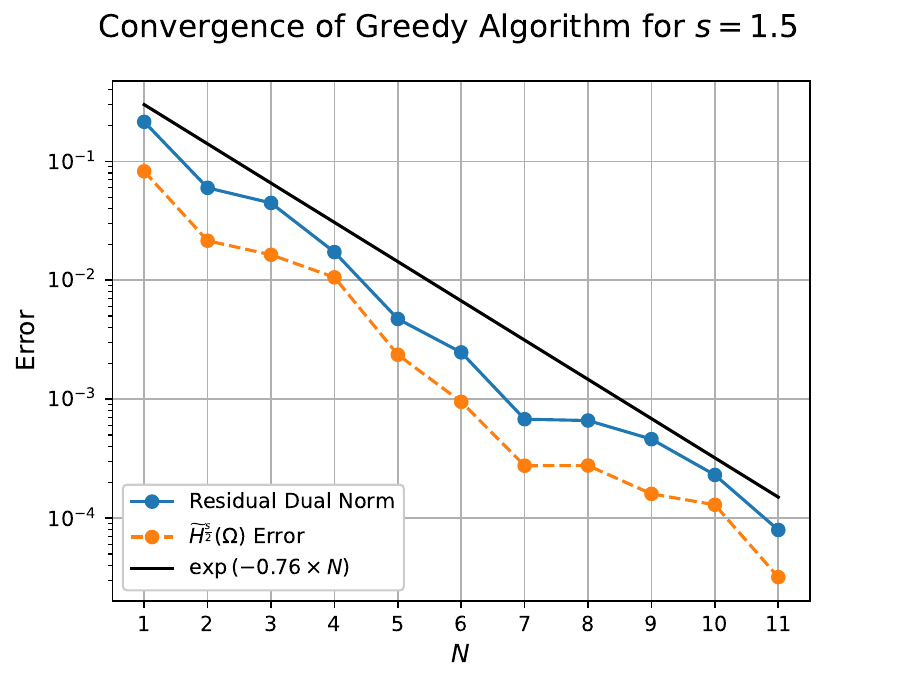}
    \caption{$s=1.5$: $e^{-0.76 n}$.}
    \label{subfig:greedy1_15}
  \end{subfigure}
  \hfill
  \begin{subfigure}{0.32\textwidth}
    \includegraphics[width=\textwidth]{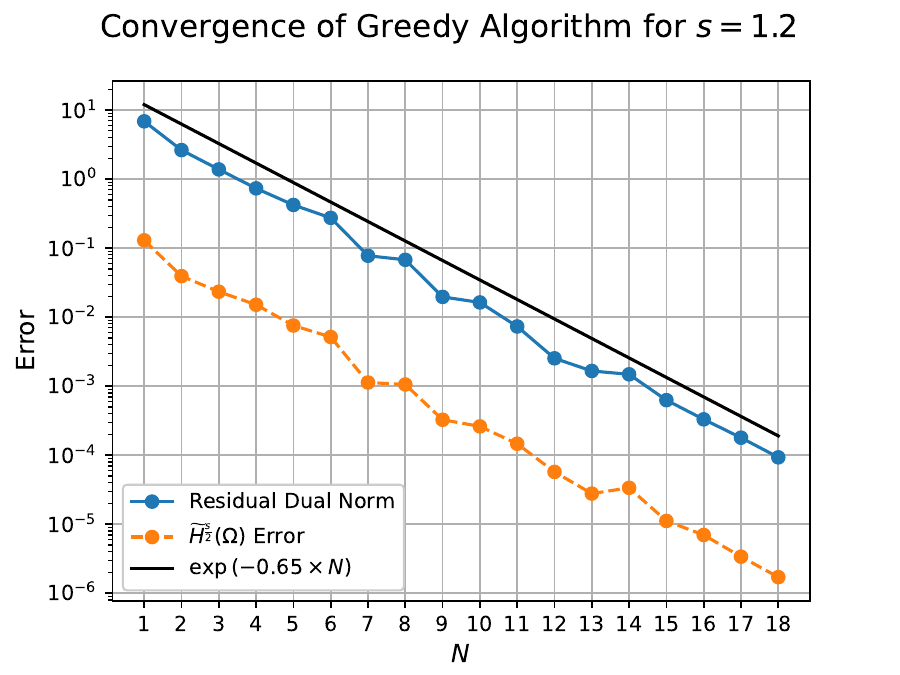}
    \caption{$s=1.2$: $e^{-0.65 n}$.}
    \label{subfig:greedy1_12}
  \end{subfigure}
  \caption{Convergence history of the Greedy Algorithm for the PFPDE \eqref{eq:greedyCase1} for $s\in\{1.8,1.5,1.2\}$.}
  \label{fig:greedy2}
\end{figure}

%-------------------------------------------------------------------------------
\subsubsection{Speedup}
%-------------------------------------------------------------------------------
In addition to the improvement achieved by reducing the number of unknowns in the reduced system, there is also another possible source for reducing computational times, namely the fact that the fractional problem is non-local, i.e., the stiffness matrix is densely populated (and non-symmetric). 

To this end, we consider the above Case 2 and compare computational times of the detailed versus the reduced system and determine the speedup. The results are shown in \Cref{tab:timeComp}, where we compare the computational time for the resolution of a single instance of the fractional differential equation. The total achieved speedup is about $3\times 10^5$, mostly due to the time required to assembly the matrix for the linear system in the full-order model. If we only compare the times for solving the linear system (the affine decomposition shown in \Cref{sec:rb}
could be used to avoid repetitive computation of the matrix), the achieved speed up is approximately $532$. The computational times were obtained by solving the fractional differential equation on a uniform mesh with $2^8=256$ elements and using $11$ elements in the reduced basis on a 2023 MacBook Air with an Apple M2 processor and 24 GB of RAM.
%---------------------------------------------------------------------------------------------------------------
\begin{table}[!htb]
	\begin{tabular}{r|r|r|r} 
	& FEM & RB & Ratio\\ \hline
		DoF 			& 255 	& 11			& 23  \\
		CPU 	& 10 ms 	& 18.8 $\mu$s	& 532 
	\end{tabular}
	\hfil
	\begin{tabular}{l|r}
		\multicolumn{2}{c}{Offline}\\ \hline
		Affine terms  	& 6 \\ \hline
		Assembly		& 5.8s \\ \hline
		RB selection	& 161.6 s
	\end{tabular}
	\caption{Comparison of the FEM and the RBM for \Cref{prob:nonConstCoeff} (with coefficients as in \cref{eq:greedyCase1} and an arbitrary choice of the parameter $\mu$).}
	\label{tab:timeComp}
\end{table}
%---------------------------------------------------------------------------------------------------------------

%-------------------------------------------------------------------------------------------------------
\section{Conclusions and Outlook}
\label{sec:conclusion}
%-------------------------------------------------------------------------------------------------------
In this paper, we considered fractional boundary value problems involving the Riemann-Liouville fractional derivative of order $s\in (1,2)$ with non-constant diffusion coefficients. We derived a variational formulation, proved its well-posedness and introduced a FE discretization mainly following a standard route. By parameterizing diffusion, reaction and right-hand side, we derived a parameterized fractional differential equation for which we introduced model reduction by the RBM. We proved the decay of the Kolmogorov $n$-width and presented results of several numerical experiments.

Some issues for future research are more or less obvious. First, we did not consider the Caputo fractional derivative as in \cite{jin}. The reason is that the Riemann-Liouville derivative puts stronger requirements on the regularity of the resulting solution, which is a more severe challenge for model reduction. Hence, we expect even better results for the Caputo derivative following the corresponding lines in \cite{jin}. We plan to investigate fractional derivatives in space and time, also involving combinations of Riemann-Liouville and Caputo derivatives.

Given data of a subdiffusive process, the rate $s$ of subdiffusion might be unknown. In that case, one could think of using a RBM for identifying the current value of $s$. This means, however, that the order $s$ would need to be a parameter. As we have seen, the solution significantly changes when $s$ varies (and the other data stays the same), so that model reduction is expected to be tough. See, for example, \cite{bonito2020reduced} for such an example of model order reduction for spectral fractional diffusion problems. We devote this aspect also to future research.

\section{Acknowledgments}
The authors would like to thank Professors Andrea Bonito, Martin Stynes and Bangti Jin
for their helpful comments and remarks on an initial version of the present manuscript.
The work on this paper has been funded by the \emph{Federal Ministry for Economic Affairs and Energy of Germany} (BMWE -- Bundesministerium für Wirtschaft und Energie der Bundesrepublik Deutschland).
\printbibliography
\end{document}